\documentclass[12pt]{article}
\usepackage[dvipsnames]{xcolor}
\pdfoutput=1

\usepackage{authblk}
\author[1]{Alvaro de Diego Unanue}
\author[2]{Gary Froyland}
\author[1]{Oliver Junge}
\author[3]{P\'eter Koltai}

\affil[1]{Department of Mathematics, TUM School of Computation, Information and Technology, Technical University of Munich, 80333 Munich, Germany}
\affil[2]{School of Mathematics and Statistics,
University of New South Wales, Sydney NSW 2052, Australia}
\affil[3]{Department of Mathematics,  University of Bayreuth, 95440 Bayreuth, Germany}

\usepackage[utf8]{inputenc}
\usepackage[T1]{fontenc}
\usepackage{lmodern}
\usepackage[english]{babel}
\usepackage{graphicx}
\usepackage[export]{adjustbox}
\usepackage{geometry}
\usepackage{enumerate}
\usepackage{amsmath}
\usepackage{amssymb}
\usepackage{oldgerm}
\usepackage{amsthm}
\usepackage{mathtools}
\usepackage{float}
\usepackage{csquotes}
\usepackage[normalem]{ulem}
\usepackage{algorithm}

\usepackage{color}
\usepackage[pdffitwindow=false,
            plainpages=false,
            pdfpagelabels=true,
            pdfpagemode=UseOutlines,
            pdfpagelayout=SinglePage,
            bookmarks=false,
            colorlinks=true,
            hyperfootnotes=false,
            linkcolor=blue,
            urlcolor=blue!30!black,
            citecolor=green!50!black]{hyperref}

\usepackage{cleveref}

\newtheorem{thm}{Theorem}[section]
\newtheorem{lem}{Lemma}[section]

\newtheorem{rem}{Remark}[section]
\newtheorem{prop}{Proposition}[section]

\newcommand\torpdf[2]{\texorpdfstring{#1}{#2}}

\newcommand\R{\mathbb{R}}
\newcommand\spn{\textrm{span}}
\DeclareMathOperator{\grad}{grad}
\newcommand\<{\langle}
\renewcommand\>{\rangle}

\newcommand\dynamic[1]{{#1}^D}
\newcommand\subsetname{A}

\newcommand{\sobolevdual}{W^{1,p}_0(M)^*}

\title{A dynamic \torpdf{$p$}{p}-Laplacian}

\begin{document}
\maketitle
\begin{abstract}
    We generalise the dynamic Laplacian introduced in (Froyland, 2015)
    to a dynamic $p$-Laplacian, in analogy to the generalisation of the standard $2$-Laplacian to  the
    standard $p$-Laplacian for $p>1$. Spectral properties of the dynamic Laplacian are connected to the geometric problem of finding ``coherent'' sets with persistently small boundaries under dynamical evolution, and we show that the dynamic $p$-Laplacian shares similar geometric connections. 
    In particular, we prove that the first eigenvalue of the dynamic $p$-Laplacian with Dirichlet boundary conditions exists and converges to a dynamic version of the Cheeger constant introduced in (Froyland, 2015)
    as~$p\rightarrow 1$.
    We develop a numerical scheme to estimate the leading eigenfunctions of the (nonlinear) dynamic $p$-Laplacian, and through a series of examples we investigate the behaviour of the level sets of these eigenfunctions. 
    These level sets define the boundaries of sets in the domain of the dynamics that remain coherent under the dynamical evolution.
\end{abstract}

\clearpage

\section{Introduction}

Time-dependent or non-autonomous dynamical systems frequently exhibit complex, turbulent motion;  prominent examples are geophysical flows in the atmosphere and ocean, industrial fluid mixers, and the dynamics of biofluids.
These flows often display heterogeneous behaviour, i.e.\ there exist regions of fluid that are being \emph{coherently} transported and slowly dispersed relative to the surrounding fluid.  
There is a large literature on the identification of coherent structures of one sort or another.
Our interest is primarily in \emph{geometric} characterisations of mixing.

In \cite{froyland_dynamic_2015}, Froyland proposed to characterise a \emph{coherent set} by a small average \emph{evolved} boundary area to volume ratio.
If this ratio is high for a particular region of fluid, it is likely that this fluid region becomes highly filamented over a finite flow time and therefore is strongly susceptible to small-scale diffusive processes.
Indeed, tracking the growth in irregularity of fluid interfaces is well-established as a means of quantifying mixing properties of fluids \cite{aref1984,sturman2006,haller2013coherent,aref2017}.
Such a characterisation extends classical (non-dynamic) notions from isoperimetric theory, whereby the \emph{Cheeger ratio} of a set is the ratio of its boundary area (boundary codimension-one volume) to its enclosed volume;  formally, the infimand on the right-hand side of~\eqref{eq:cheeger}.
In the characterisation of \cite{froyland_dynamic_2015}, coherent sets have a small \emph{dynamic Cheeger ratio}, defined formally in~\eqref{eq:dynCheegerratio}. 
The problem of finding sets with small dynamic Cheeger ratio can be transformed into finding eigenfunctions of a certain Laplace--Beltrami operator, called \emph{dynamic Laplacian}~\cite{froyland_dynamic_2015}.

To set up our contributions in this work, we recap some classical isoperimetric theory. Throughout the paper (except for one of the examples), we consider a compact 
$d$-dimensional manifold $M\subset \mathbb{R}^d$ with nonempty Lipschitz boundary. Classically, the \emph{Cheeger constant} of $M$ is defined as \cite{chavel_isoperimetric_2001}
\begin{equation}
\label{eq:cheeger}
    h(M) := \inf_{\subsetname\subset M}
        \frac{\ell_{d-1}(\partial \subsetname)}{\ell_{d}(\subsetname)},
\end{equation}
where the infimum is taken over all open submanifolds $A$  of $M$ with $C^\infty$ boundary and compact closure and $\ell_d$ denotes  $d$-dimensional volume. 

If a set $\subsetname$ has a Cheeger ratio of $h(M)$, it is called
a \emph{Cheeger set}. Note that a Cheeger set does not have to have $C^\infty$ boundary and might intersect the boundary of~$M$. 

The Federer--Fleming theorem~\cite{federer1960normal, maz1960classes} provides a tight connection between $h(M)$ and the functional \emph{Sobolev constant} defined by an infimum taken over $C^\infty$ functions $u:M\to \mathbb{R}$ with compact support\footnote{For a vector-valued function $v$, we denote by $\| v \|_p$ the $L^p$ norm of the pointwise Euclidean length $|v|$ of~$v$.}
\begin{equation}
    s(M):=\inf_{u\in C^\infty_c(M)}\frac{\|\nabla u\|_1}{\|u\|_1},
    \label{eq:sobolev-constant}
\end{equation}
namely $h(M)=s(M)$ \cite[Remark 2.1]{leonardi2015overview}.
One may think of the supports of an infimising sequence of functions $u$ in \eqref{eq:sobolev-constant} defining an infimising sequence of sets $\subsetname$ in~\eqref{eq:cheeger}.
Replacing the 1-norm by the 2-norm in \eqref{eq:sobolev-constant} produces the related minimisation problem
\begin{equation}
    \label{eq:eigval}
    \lambda_2 = \inf_{u\in C^\infty_c(M)}\frac{\|\nabla u\|_2^2}{\|u\|_2^2}.
\end{equation}
In fact, $\lambda_2>0$ is the leading eigenvalue\footnote{Later we will denote by $\lambda_p^{(k)}$ the $k^{\rm th}$ eigenvalue of the $p$-Laplacian.  For now, we suppress the superscript when $k=1$.} of the (negative) 2-Laplace operator on $M$ with homogeneous Dirichlet boundary conditions, i.e.\ $u=0$ on $\partial M$ \cite[Rayleigh's Theorem]{chavel1984eigenvalues}.  
A connection between the two problems is given by the well-known \emph{Cheeger inequality} 
\begin{equation}
    \label{eq:cheeger_ineq}
h(M)\le 2\sqrt{\lambda_2}\,.
\end{equation}
Informally one may think of the Laplace eigenfunction $u_2$  corresponding to $\lambda_2$ as a smoothed version of the limit of an infimising sequence of functions $u$ in~\eqref{eq:sobolev-constant}, and one of the superlevel sets of $u_2$  as a smoothed version of the limit of an infimising sequence of sets $\subsetname$ in~\eqref{eq:cheeger}. 

Dynamic counterparts of these constructions (for the case of Neumann boundary conditions) have been developed in \cite{froyland_dynamic_2015},  and are recalled in \cref{sec:prelims}.  The particular expressions for the Dirichlet boundary condition case, \eqref{eq:dynCheeger}--\eqref{eq:dynamic_Laplace}, may be found in~\cite{froyland_robust_2018}.  
As in the static case, a suitable  superlevel set of the leading eigenfunction  of the dynamic Laplacian yields a smoothed version of the set of interest, namely a set with small averaged isoperimetric ratio, i.e.\ a \emph{coherent set}.  
Multiple coherent sets can be found directly from suitable  superlevel sets of eigenfunctions at higher eigenvalues of the dynamic Laplacian~\cite{froyland_dynamic_2015,BaKo17,froyland_robust_2018,FrRoSa19}.  Computationally, the eigenproblem for the dynamic Laplacian can efficiently be solved by the finite element method \cite{froyland_robust_2018}, leading to a scheme for which only trajectory data is required. 

While the use of the eigenfunctions is computationally advantageous on the one hand, it is not \emph{a priori} clear how well the level sets of these eigenfunctions approximate solutions of the original problem \eqref{eq:cheeger}.
On the other hand, while it is desirable to directly solve the original formulation using the 1-norm,  the solutions to this problem are in general not smooth and typically certain regularisations of the problem have to be considered~\cite{feng_analysis_2003}.
Instead of regularising the $p=1$ eigenproblem, one can try to solve a (nonlinear) eigenproblem for a $p$-norm with~$ 1\lesssim p\leq 2 $;  this is the approach we pursue here. 
We show that the 2-norm constructions of \cite{froyland_dynamic_2015} can be generalised to $p$-norm versions, obtaining a $p$-norm version of the dynamic Sobolev constant and a dynamic $p$-Laplacian.
We prove a dynamic $p$-Cheeger inequality for the leading eigenvalue of the (nonlinear) dynamic $p$-Laplacian (Theorem~\ref{thm:cheegerdynamic}), generalising the dynamic Cheeger inequality of~\cite{froyland_dynamic_2015}.
In Theorem~\ref{thm:cheegerconv} we show that in the limit $p\to 1$ this dynamic $p$-Cheeger inequality interpolates to the dynamic Federer--Fleming equality of~\cite{froyland_dynamic_2015}.
Thus, for $p$ close to 1, one expects most level sets of the dynamic $p$-Laplacian to have a small dynamic Cheeger ratio (relative to other sets) and therefore to represent a relatively coherent set.

The numerical computation of the eigenfunctions of the (dynamic) $p$-Laplacian becomes more challenging as $p$ comes closer to $1$.  
We employ a scheme proposed by Yao and Zhou \cite{yao_numerical_2007} for a certain general class of nonlinear eigenproblems. 
Our numerical experiments in \cref{sec:experiments} allow us to compare coherent sets that result from  superlevel sets of eigenfunctions of the dynamic $p$-Laplacian for $p$ close to 1 with those for $p=2$.  
In all cases, we find the dynamic $p$-Laplacian eigenfunctions have profiles closer to indicator functions than the smooth profile of the dynamic $2$-Laplacian. 
Despite this, the difference between the level sets of the eigenfunctions that minimise their Cheeger ratios in the $1\lesssim p$ and $p=2$ cases is small, indicating that the coherent sets  obtained through level sets of the dynamic 2-Laplacian are indeed accurate estimates of solutions of the dynamic Cheeger problem.

\section{Preliminaries}
\label{sec:prelims}
\subsection{The dynamic isoperimetric problem}
As alluded to in the introduction, dynamic versions of \eqref{eq:cheeger} and \eqref{eq:sobolev-constant} can be constructed \cite{froyland_dynamic_2015} (see \cite{froyland_robust_2018} for the case of Dirichlet boundary data as used below).  The underlying dynamical system is defined by a volume-preserving diffeomorphism\footnote{In \eqref{eq:dynamic_Laplace} below, from a classical point of view, we need to assume $T$ to be a $C^2$-diffeomorphism. Since we always view $\Delta$ as mapping into a dual space (see \eqref{eq:formulaplaplacian} and for the dynamic version~\eqref{eq:dyn_p_Laplace}), however, $C^1$ regularity of $T$ is actually sufficient.} $T:M\rightarrow M$, which in applications is often given by the flow map of a conservative vector field over some finite time span. We denote by $T_*:u\mapsto u\circ T^{-1}$, $u\in L^p(M)$, the \emph{pushforward operator} corresponding to $T$. Since we assume $T$ to be volume-preserving, $T_*$ coincides with the \emph{transfer} (or \emph{Perron--Frobenius}) {operator} of $T$. 
Its adjoint is $T^*$, the Koopman operator, with $T^*:u\mapsto u\circ T$, $u\in L^q(M)$.

For a single application of $T$, the 
\emph{dynamic Cheeger ratio} of a subset $A\subset M$ is 
    \begin{equation}
    \label{eq:dynCheegerratio}
            \frac{\ell_{d-1}(\partial \subsetname) +
                  \ell_{d-1}(T(\partial \subsetname))}{2\ell_d(\subsetname)},
    \end{equation}
    and the \emph{dynamic Cheeger constant} of $M$ is 
    \begin{equation}
    \label{eq:dynCheeger}
        \dynamic{h}(M, T):=
            \inf_{\subsetname\subset M}
            \frac{\ell_{d-1}(\partial \subsetname) +
                  \ell_{d-1}(T(\partial \subsetname))}{2\ell_d(\subsetname)},
    \end{equation}
    where the infimum is taken over open submanifolds $\subsetname$ with compact closure and $C^{\infty}$ boundary. 
    Similarly, the \emph{dynamic Sobolev constant}  is 
    \begin{equation}
    \label{eq:dynamic_sobolev_constant}
    \dynamic{s}(M,T):=
                  \inf_{u\in C^\infty_c(M)}
                 \frac{\|\nabla u \|_1 +
                       \|\nabla T_*u\|_1}
                      {2\|u\|_1}.
     \end{equation}

As in the static case, one has a dynamic Federer--Fleming theorem, i.e.\  equality of these two quantities (\cite{froyland_dynamic_2015} for the Neumann case, \cite{froyland_robust_2018} for the Dirichlet case):
$\dynamic{h}(M,T) = \dynamic{s}(M,T)$.
Modifying the norm in problem \eqref{eq:dynamic_sobolev_constant} as in the static case yields the minimisation problem
\begin{equation}
    \dynamic{\lambda}_2 =
    \inf_{u\in C^\infty_c(M)}  
    \frac{\|\nabla u \|_2^2 + \|\nabla T_*u\|_2^2} {2\|u\|_2^2}.
\end{equation}
Also like in the static case, $\dynamic{\lambda}_2 > 0$ is the leading eigenvalue of the (negative) \emph{dynamic Laplace operator} \cite{froyland_dynamic_2015}
\begin{equation}
\label{eq:dynamic_Laplace}
    \dynamic{\Delta} := \frac12 \left(\Delta + T^*\Delta T_*\right)
\end{equation}
on $M$ subject to homogeneous Dirichlet boundary conditions \cite{froyland_robust_2018}.  A suitable  superlevel set of an eigenfunction of $\dynamic{\Delta}$ at $\dynamic{\lambda}_2$ yields a smoothed version of the limit of an infimising sequence of sets $\subsetname$ in \eqref{eq:dynCheeger}, i.e.\ a coherent set.

\subsection{1-norm minimisation versus 2-norm minimisation}
\label{sec:effect_of_relaxation}

In order to illustrate the effect of the transition from the $1$-norm to the $2$-norm in \eqref{eq:sobolev-constant}, we consider a minimiser of the static problem on the unit square~$M=[0,1]^2$ (more precisely, the limit
of an infimising sequence).
It can be shown (see e.g. \cite{kawohl_isoperimetric_2003}) that
almost all of its level sets coincide with the minimiser of the Cheeger problem. 
For this simple domain, one can explicitly calculate the Cheeger constant and the Cheeger set (e.g.\ \cite{horak_numerical_2011}):
\[
    h([0,1]^2) = \frac{4-\pi}{2-\sqrt{\pi}}\approx 3.772.
\]
The Cheeger set is a square with rounded corners of radius $1/h([0,1]^2) \approx 0.265$ touching the boundary; see e.g.\ \cite[Remark 5]{kawohl_isoperimetric_2003} and Figure \ref{fig:L1-minimizer}.

\begin{figure}[H]
\centering
    \includegraphics[width=0.25\textwidth]{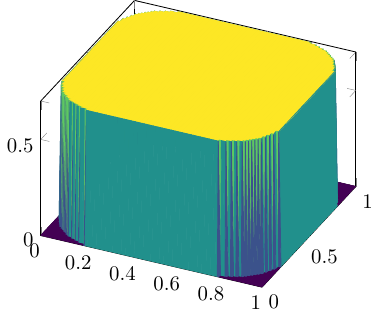}\qquad
    \includegraphics[width=0.22\textwidth]{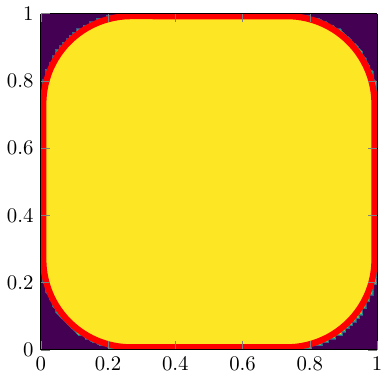}
\caption{Left: minimiser of \eqref{eq:sobolev-constant}. Right: corresponding minimiser of \eqref{eq:cheeger}. The red curve delimits the Cheeger set. }
\label{fig:L1-minimizer} 
\end{figure}
In contrast, if we alter the $1$-norm to the $2$-norm, the corresponding minimiser $u(x,y)=\sin(\pi x y)$ is quite smooth and its level sets vary strongly. 
Nevertheless, at least in this simple example, the optimal  superlevel set -- with boundary shown in red in Figure ~\ref{fig:L2-minimizer} -- is close to the minimiser of \eqref{eq:cheeger}.
Moreover, the Cheeger ratio of this  superlevel set is approximately $3.890$, which is close to the exact minimal value of around $3.772$.

\begin{figure}[H]
\centering
    \includegraphics[width=0.25\textwidth]{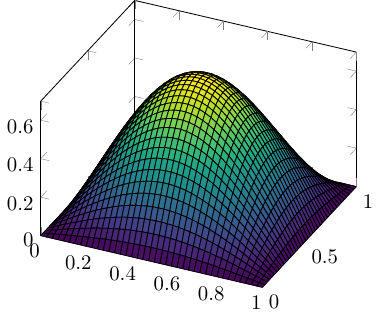}\qquad
    \includegraphics[width=0.22\textwidth]{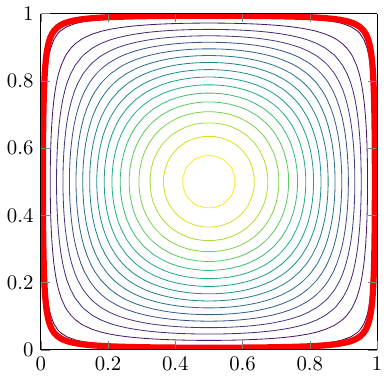}
\caption{Graph of $u(x,y)=\sin(\pi x y)$, the minimiser of $\frac{\|\nabla u\|_2^2}{\|u\|_2^2}$, $u|_{\partial M} = 0$ (left) and the level sets (right). The red curve is the level set of $u$ with optimal Cheeger ratio.}
\label{fig:L2-minimizer}
\end{figure}

Our aim in the sequel is to consider the norms $\|\cdot\|_p$ as  $p\in (1, 2)$ and to investigate what happens for~$p\rightarrow 1$.
Ideally, the minimisers become flatter away from the optimal level set, while the level sets get closer to the optimum level set, so that the optimal level sets can be more robustly extracted. 
We also anticipate that the infimum of the Rayleigh quotient with respect to the $p$-norms converges to~$h(M)$.
In the classical (static) setting the above behaviour of the level sets and the convergence of the Rayleigh quotient can be proven (see e.g. \cite{kawohl_isoperimetric_2003}). Further, the \enquote{flatness} of the eigenfunctions for $p$ close to $1$ can be observed in numerical solutions (see e.g. \cite[Figure 8]{horak_numerical_2011}). However, no such attempt has been made so far for a functional involving dynamics. 

\subsection{The \torpdf{$p$}{p}-Laplacian}
\label{sec:plaplacian}

In this section we recall results on the consequences of replacing
the $1$-norm by the $p$-norm in \eqref{eq:sobolev-constant}.
The minimiser can then be found by computing eigenfunctions of the 
so-called \emph{$p$-Laplacian} (see e.g. the introduction of \cite{lindqvist_nonlinear_2008}). In the limit $p\rightarrow 1$, these eigenfunctions allow us to recover solutions
of the Cheeger problem \eqref{eq:cheeger} (see \Cref{thm:plaptoone} below). 
We will give a short overview of the definition of the $p$-Laplacian and some
known results that link it to the Cheeger problem. 
In what follows we always assume that 
$1 < p < 2$ and denote by $q$ the conjugate of $p$
satisfying~$\smash{ 1/p + 1/q = 1 }$.

Let $W^{1,p}_0(M)$ denote the classical Sobolev space of once weakly differentiable $L^p$ functions on $M$ that are zero on the boundary and have weak first derivatives in~$L^p$. Consider the functional~$F:W_0^{1,p}(M)\to\mathbb R$,
\[
    F(u):= \|\nabla u\|_p^p.
\]
Its (Fréchet) derivative $F'(u)$ at some $u\in W^{1,p}_0(M)$ is an element of the dual space of $W^{1,p}_0(M)$, which is denoted by $\sobolevdual$.
The action of $F'(u)$ on a function $v\in W^{1,p}_0(M)$ can be calculated by $F'(u)v = \lim_{h\to 0} \frac1h (F(u+hv) - F(u))$ to be
\[
F'(u)v = p\int_M \nabla u |\nabla u|^{p-2}\nabla v.
\]
The (classical) $p$-Laplacian $\Delta_p:W_0^{1,p}(M)\rightarrow \sobolevdual$ 
is defined as the nonlinear operator ${\Delta_p u :=-\frac1pF'(u)}$, thus
\begin{equation}
    \label{eq:formulaplaplacian}
    (\Delta_p u)v =
    -\int_M \nabla u |\nabla u|^{p-2}\nabla v.
\end{equation}
If $u$ is sufficiently regular, then the functional $\Delta_p u$ can even be represented by a function; in a similar fashion to Riesz duality between $L^p$ and~$L^q$.   
    Thus, identifying the integral operator with its integral kernel, we sometimes write
\[
    \Delta_p u = \mathrm{div}(|\nabla u|^{p-2}\nabla u)
\]
for $u$ regular enough (meaning the right-hand side is in~$L^q(M)$). This \enquote{strong form} of the operator $\Delta_p$ will be intuitive for the eigenvalue problems we introduce below.  For rigorous derivations we always consider the \enquote{weak form}~\eqref{eq:formulaplaplacian}. For $p=2$, the $p$-Laplacian coincides with the standard Laplacian.

We will mainly be  interested in the variational problem
\begin{equation}
    \min F(u) \quad\text{such that} \quad G(u):= \|u\|_p^p = 1.
    \label{eq:plapopt}
\end{equation}
The Euler--Lagrange equation of \eqref{eq:plapopt} is
\[
    F'(u) = \lambda G'(u).
\]
As the derivative of $G$ is $G'(u)v  = p\int_M u|u|^{p-2}v$, the expression immediately above becomes
\begin{equation}
    \int_M\nabla u|\nabla u|^{p-2}\cdot \nabla v =
    \lambda \int_M u|u|^{p-2}v\label{nlevp2}
    \quad \text{for all } v \in W_0^{1,p}(M),
\end{equation}
or in its strong version
\begin{equation}
    -\Delta_p u = \lambda u|u|^{p-2} \label{nlevp}.
\end{equation}
If a pair $(\lambda, u)\in\R\times W_0^{1,p}(M)$ solves \eqref{nlevp2}, we call $\lambda$ an eigenvalue of $-\Delta_p$ and $u$ an eigenfunction.  In contrast to the linear case $p=2$, where a plethora of results is available, the set of eigenvalues of the $p$-Laplacian is not fully known \cite{parini_second_2010, lindqvist_nonlinear_2008}.  One can construct a subset of it by a min-max-principle based on the Krasnoselskii genus \cite{parini_second_2010}, but it is not known whether these exhaust the spectrum of $\Delta_p$. 

In our case, we will be mostly concerned with the smallest
eigenvalue, which corresponds to the minimum of \eqref{eq:plapopt}.  
Since $F$ and $G$ have the same order $p$ of homogeneity,  one can equivalently search for a minimum of the functional $J:W_0^{1,p}(M)\backslash\{0\}\to\R$,
\[
    J(u) = \frac{F(u)}{G(u)}.
\]
\begin{thm}{\cite[Lemma 5]{lindqvist_nonlinear_2008}}
    Let $d\geq2$, $1<p<2$ and $M \subseteq \mathbb{R}^d$ a
    bounded manifold of dimension $d$ with Lipschitz boundary. 
    Then
    \[
        \lambda_p := \inf_{u\in W_0^{1,p}(M) \setminus \{0\}}  J(u)
    \]
    is positive and achieved by a function
    $u_p\in W_0^{1,p}(M)$ satisfying
    \eqref{nlevp2} with $\lambda=\lambda_p$. 
    \label{thm:variationalplap}
\end{thm}
\begin{thm}{\cite[Lemma 5]{lindqvist_nonlinear_2008}} 
\label{thm:positivity}
Let $d\geq 2$, $1<p<2$ and $M\subset \mathbb{R}^d$ a bounded 
manifold of dimension $d$ with Lipschitz boundary. If $u\in W^{1,p}_0(M)$ is a solution of \eqref{eq:plapopt} then $u$ is also a solution of \eqref{nlevp2}. Further, a solution $u$ of \eqref{eq:plapopt}
either fulfils $u>0$ or $u<0$ in the interior of $M$, i.e. it does not change sign. 
\end{thm}
\begin{thm}{\cite[Theorem 6]{lindqvist1990equation}} 
\label{thm:plapuniqueness}
Let $d\geq 2$, $1<p<2$ and $M\subset\mathbb{R}^d$ a bounded manifold of dimension $d$ with Lipschitz boundary and $F(u) := \frac1p\|\nabla u\|_p^p$. Then all minima of $F$ are scalar multiples of each other. 
\end{thm}
The preceding theorems show that there is a unique non-negative normalised solution to \eqref{eq:plapopt}, satisfying \eqref{nlevp2} with $\lambda= \lambda_p$. 
We call $\lambda_p$ the \emph{first eigenvalue} of $-\Delta_p$ and $u_p$ the \emph{first eigenfunction} of~$-\Delta_p$ on $M$.

We are particularly interested in
the following properties that link
$u_p$ and $\lambda_p$ to the Cheeger problem 
in the limit~$p\rightarrow 1$.
\begin{thm} 
\label{thm:plaptoone} Let $M\subseteq \mathbb{R}^d$ be a bounded manifold of dimension $d$ with Lipschitz boundary. Then in the limit
$p\rightarrow 1$ 
\begin{enumerate}[(a)]
\item  
\label{thm:plaptoone_a}
the first eigenvalue $\lambda_p$ converges to~$h(M)$;
\item  \label{thm:plaptoone_b} 
after possibly passing to a subsequence, 
the first eigenfunction of 
$-\Delta_p$  converges in $L^1$
to a function $u\in BV(M)$ (the space of functions of bounded variation on~$M$) such that the superlevel sets $A_t:=\{x\in M\mid u(x)>t\}$
are either null sets or Cheeger sets of $M$ for almost all~$t>0$; 
\item \label{thm:plaptoone_c} if $M$ has a unique Cheeger set $\subsetname$, then $u_p$ converges
in $L^1$ to a scalar multiple of the characteristic function $\chi_\subsetname$.
\end{enumerate}
\end{thm}
\begin{proof}[Proof sketch.]
    For (\ref{thm:plaptoone_a}) see e.g. \cite[Corollary 6]{kawohl_isoperimetric_2003}.
    For (\ref{thm:plaptoone_b})
    one can follow 
    \cite[Theorem 8, Remark 10]{kawohl_isoperimetric_2003} to 
    show that by compactness of the embedding
    $BV(M)\rightarrow L^1(M)$ there must be a 
    subsequence converging to some $u$ in $L^1$, for which 
    one can deduce from (\ref{thm:plaptoone_a}) that it must fulfil
    \begin{equation}
    \label{eq:cheegerintegral}
        \int_{0}^\infty \ell_{d-1} (\partial A_t) \; dt = \int_0^\infty h(M)\ell_d(A_t) \; dt.
    \end{equation}
    As $\ell_{d-1}(\partial A_t)\leq h(M) \ell_d(A_t)$, 
    this implies that 
    $\ell_{d-1}(\partial A_t)= h(M) \ell_d(A_t)$ for almost all $t$, 
    and hence $A_t$ is either a null set or a Cheeger set. 
    For (\ref{thm:plaptoone_c}) we refer to \cite[Remark 3.2.]{parini_introduction_2011},
    where it is noted that one can show $\ell_{d-1}(\partial A_t)= h(M) \ell_d(A)$ for \emph{all} $t$. By uniqueness of the Cheeger set, this means that $A_t=\subsetname$ for all $t>0$ where $A_t$ is not a null set 
    and that this implies that $u$ is a suitably scaled characteristic function on a 
    Cheeger set of $M$. 
\end{proof}
In \cref{ssec:dynamiclimitpto1} we will show
an analogous statement to \Cref{thm:plaptoone}(\ref{thm:plaptoone_a}) for the dynamic $p$-Laplacian introduced in \cref{sec:dynamicplaplacian}. 
\Cref{thm:plaptoone}(\ref{thm:plaptoone_b}) and (\ref{thm:plaptoone_c}) 
are investigated numerically in \cref{sec:experiments}, in particular we study the improvement of the Cheeger ratios of the level 
sets as~$p\to 1$.

\section{A dynamic \torpdf{$p$}{p}-Laplacian}
\label{sec:dynamicplaplacian}
\subsection{Motivation}
We now aim to generalise the $p$-Laplacian to a
dynamic $p$-Laplacian which shares the same
connections to the dynamic Cheeger problem as
the classic $p$-Laplacian does to the classic Cheeger problem.
Recall that we consider a volume-preserving diffeomorphism $T:M\rightarrow M$,  use the pushforward operator $T_*:L^p(M)\rightarrow L^p(M)$, $T_*u := u\circ T^{-1}$, 
and its adjoint $T^*:L^q(M)\rightarrow L^q(M)$, $T^*u := u\circ T$, the Koopman operator.  
Note that for $u\in W^{1,p}_0(M)$ we have $T_*u\in W^{1,p}_0(M)$
by compactness of $M$ so we can use $T_*$ as an 
operator from $W^{1,p}_0(M)$ to~$W^{1,p}_0(M)$. 

Equipped with this we define a dynamic version $\dynamic{F}:W^{1,p}_0(M)\rightarrow \R$  of $F$ as 
\begin{align}
    \dynamic{F}(u) &:= \frac12 \left( \|\nabla u\|_p^p +\|\nabla T_* u\|_p^{p} \right)
\end{align}
or equivalently
\[
    \dynamic{F} = \frac12(F + F\circ T_*).
\]
It now follows directly from the differentiability of $F$ and boundedness of $T_*$ that $\dynamic{F}$ is differentiable with 
\begin{align}
    \label{eq:derivativeFbar}
     (\dynamic{F})'(u)v &= \frac12 \bigl( F'(u)v + F'(T_*u)T_*v \bigr)\\
    &= \frac12p\bigl((\Delta_pu)v+ (T^*\Delta_pT_*u)v\bigr).
    \label{eq:derivativeFbar2}
\end{align}
To get from \eqref{eq:derivativeFbar} to \eqref{eq:derivativeFbar2} we see the Koopman operator $T^*$ as the purely linear
algebraic adjoint of $T_*$, meaning that $T^*$ maps some $f\in (W^{1,p}_0(M))^*$ to $f\circ T_*$. This directly implies 
$F'(T_*u)T_*v = (F'(T_*u)\circ T_*)v = (T^*F'(T_*u))v$. While this differs from the 
common definition of $T^*$ as a map on $L^q(M)$ defined by $f\mapsto f\circ T$, the two definitions coincide on $L^q$ 
under the usual identification of some $f\in L^q$ with 
the functional  $(u\mapsto \int_M uf)\in (W^{1,p}_0(M))^*$, as by 
volume preservation of $T$:
\[
    \left(u\mapsto \int_M u\cdot f\right)\circ T_* =
    \left(u\mapsto \int_M (u\circ T^{-1}) \cdot f\right) = 
    \left(u\mapsto \int_M u\cdot(f\circ T)\right).
\]
This motivates the following definition for a dynamic $p$-Laplacian:
\begin{align}
    \dynamic{\Delta}_pu &:= -\frac1p(\dynamic{F})'(u) = \frac12 \left(\Delta_p u + T^{*}\Delta_pT_*u\right). 
\end{align}
Plugging in $F'(u)v = p\int_M |\nabla u|^{p-2}\nabla u\nabla v$
into \eqref{eq:derivativeFbar} we get the equivalent definition 
\begin{equation}
\label{eq:dyn_p_Laplace}
(\dynamic{\Delta}_p u)v:=\frac12\left (\int_M \nabla u |\nabla u|^{p-2}\nabla v + \int_M\nabla T_*u|\nabla T_*u|^{p-2} \nabla T_*v\right).
\end{equation}
Analogously to \Cref{thm:variationalplap}, we have:
\begin{thm}
    \label{thm:dyn_eigenvalue}
    Let $d\geq2$, $1<p<2$ and $M \subseteq \mathbb{R}^d$
    a compact $d$-dimensional manifold with 
    nonempty Lipschitz boundary. 
    Define $\dynamic{J}:W_0^{1,p}(M)\backslash \{0\}\rightarrow
    \mathbb{R}$ by
    \[
        \dynamic{J}(u):= \frac{\dynamic{F}(u)}
        {G(u)}.
    \]
    The infimum
    \begin{equation}
    \label{eq:dyn_p_Sobolev_constant}
    \dynamic{\lambda}_p (M,T):= \inf_{u\in W_0^{1,p}(M)\backslash\{0\}} \dynamic{J}(u)
    \end{equation}
    is positive and achieved by a function $\dynamic{u}_p\in W_0^{1,p}(M)$ satisfying
    \begin{equation}
    \label{eq:dynamic_eigenproblem}
        -\dynamic{\Delta}_p \dynamic{u}_p = \dynamic{\lambda}_p \dynamic{u}_p|\dynamic{u}_p|^{p-2}
    \end{equation}
    weakly.
    We call $\dynamic{\lambda}_p$ the \emph{first eigenvalue} of
    $-\dynamic{\Delta}_p$ and
    $\dynamic{u}_p$ a \emph{first eigenfunction} of $-\dynamic{\Delta}_p$.
\end{thm}
\begin{rem} \label{rem:uniq}
We do not prove a uniqueness result as in \Cref{thm:plapuniqueness} and thus do not talk about \emph{the} first eigenfunction of $-\dynamic{\Delta}_p$ in \Cref{thm:dyn_eigenvalue}. The eigenfunction $u^D_p$ is also not proven to be positive, in analogy to the positivity
result  in \Cref{thm:positivity}.
\end{rem}
\begin{proof}[Proof of \Cref{thm:dyn_eigenvalue}]
    We use the standard direct method that is
    also used for the (static) $p$-Laplacian. The
    crucial part is weak lower semicontinuity of the
    numerator of $\dynamic{J}$.
    It is well known that $u\mapsto\|\nabla u\|_p^p$
    is weakly lower semicontinuous, as it is continuous
    and convex (see e.g.\ \cite[Theorem 1.5.3]{badiale_semilinear_2011}).
    To establish weak lower semicontinuity of
    $u\mapsto \|\nabla T_* u\|_p^p$, we first show
    that for every weakly convergent series
    $u_k\rightharpoonup u$ we also have~$T_*u_k\rightharpoonup T_*u$. This is a simple
    consequence of linearity and boundedness of $T_*$:
    For any $f\in \sobolevdual$ we can define
    $\tilde f:= f\circ T_* \in \sobolevdual$,
    and hence $u_k\rightharpoonup u$ implies
    \[
        f(T_*u_k)=\tilde f(u_k) \rightarrow \tilde f(u)
                 = f(T_*u).
    \]

    This proves $T_*u_k \rightharpoonup T_*u$ and implies that
    for a weakly lower semicontinuous map $H$, the map
    $H\circ T_*$ is also weakly lower semicontinuous.
    In particular the functional $u\mapsto\|\nabla T_*u\|_p^p$
    and hence also the functional
    \[
        \dynamic{F}(u) := \frac12 \left(\|\nabla u\|_p^p +
                             \|\nabla T_*u\|_p^p\right)
    \]
    are weakly lower semicontinuous.
    Now we can proceed with a standard argument,
    as seen e.g.\ in
    \cite[Theorem 2.6.11]{badiale_semilinear_2011}:
    Clearly, $\dynamic{J}$ is bounded from below by $0$ and
    thus has a non-negative finite infimum~$\dynamic{\lambda}_p$.
    The Poincar\'e inequality establishes that
    for some positive $C$,
    \[
        G(u) = \|u\|_p^p \leq C\|\nabla u\|_p^p \leq 2C\dynamic{F}(u)
    \]
    and thus $\dynamic{\lambda}_p\geq\frac1{2C}$ is positive.
    Let $(u_k)_k$ be a minimising sequence for $\dynamic{J}$.
    We can assume without loss that $\|u_k\|_p = 1$ for all $k$, as $\dynamic{J}$ is homogeneous. The sequence is
    therefore bounded in $W_0^{1,p}(M)$ and we
    can pass to a weakly convergent
    subsequence and assume $u_k\rightharpoonup u$ for
    some~$u\in W_0^{1,p}(M)$.
    Now by the Rellich--Kondrachov theorem (see e.g. \cite[Theorem 6.3]{adams2003sobolev}) the inclusion mapping
    $W_0^{1,p}(M)\xhookrightarrow{} L^p(M)$ is compact,
    and as compact operators map weakly convergent
    sequences to strongly convergent ones, we have
    $u_k \rightarrow u$ in $L^p$ and hence~$G(u) :=\|u\|_p^p = 1$.
    Finally, by weak lower semicontinuity of $\dynamic{F}$ we have
    \[
        \dynamic{\lambda}_p \leq \dynamic{J}(u) =
        \dynamic{F}(u) \leq \liminf_k \dynamic{F}(u_k) =
        \liminf_k \dynamic{J}(u_k) = \dynamic{\lambda}_p
    \]
    and hence $\dynamic{\lambda}_p$ is achieved by some $\dynamic{u}_p:=u$.
    To show that $\dynamic{u}_p$ satisfies (\ref{eq:dynamic_eigenproblem}), note that since $\dynamic{u}_p$ is a global minimum of $\dynamic{J} = \dynamic{F}/G$ we have $(\dynamic{J})'(\dynamic{u}_p)= 0$, so for all $v\in W_0^{1,p}(M)$:
    \begin{align*}
        0 &= (\dynamic{J})'(\dynamic{u}_p)(v)  \\
          &= \frac{1}{G(\dynamic{u}_p)^2}(G(\dynamic{u}_p)(\dynamic{F})'(\dynamic{u}_p)(v) -
                           \dynamic{F}(\dynamic{u}_p)G'(\dynamic{u}_p)(v)) \\
          &= (\dynamic{F})'(\dynamic{u}_p)(v)- \dynamic{\lambda}_pG'(\dynamic{u}_p)(v).
    \end{align*}
    Now the last expression is the weak form of
    $-p\dynamic{\Delta}_p \dynamic{u}_p = p\dynamic{\lambda}_p \dynamic{u}_p|\dynamic{u}_p|^{p-2}$,
    which shows the claim.
\end{proof}

\subsection{Behaviour as \torpdf{$p\rightarrow 1$}{p goes to ll1}}
\label{ssec:dynamiclimitpto1}
\Cref{thm:plaptoone}(\ref{thm:plaptoone_a}) states that  the first eigenvalue $\lambda_p$ of $-\Delta_p$ with Dirichlet boundary conditions converges to $h(M)$ as $p\rightarrow 1$ in the static situation.
We prove the analogous statement for the \emph{dynamic} $p$-Laplacian with Dirichlet boundary conditions. 
In the proof, an inequality that plays a crucial role is
\[
    \lambda_p \geq \left(\frac{h(M)}{p}\right)^p,
\]
which  is called the \emph{Cheeger} inequality in the case $p=2$ (see \cite{gunning_lower_1970} for
$p=2$ and \cite[Appendix]{lefton_numerical_1997} or 
\cite[Theorem~3]{kawohl_isoperimetric_2003} for general $p$).
We prove this inequality for the \emph{dynamic} $p$-Laplacian with Dirichlet boundary conditions.
For the readers' convenience, we state two classical results that we use in the proof.
\begin{thm}[Co-area formula] Let $M\subseteq{\mathbb{R^d}}$ be a compact $d$-dimensional submanifold with Lipschitz boundary and~$u\in C_0^{\infty}(M)$. Then
       \[ \int_M |\nabla u| = \int_{-\infty}^{\infty}
            \ell_{d-1}(\partial A(t))\ dt,
        \]
        where $A(t)=\{x\in M\mid u(x) > t\}$ and the boundary is relative to~$M$.
\end{thm}
\begin{proof} See e.g.\ \cite[Theorem I.3.4 and Corollary~I.3.1]{chavel_isoperimetric_2001}. The formula is commonly stated with $u^{-1}(t)$ instead of $\partial A(t)$. For a proof that $u^{-1}(t) = \partial A(t)$ for almost all $t$ see \Cref{lem:introducingabs}.
\end{proof}

The second theorem is a special case of Cavalieri's Principle:
\begin{thm}[Cavalieri's Principle]
 Let $M\subseteq{\mathbb{R}^d}$ be a compact $d$-dimensional manifold with Lipschitz boundary and $f:M\rightarrow \mathbb{R}$ be a non-negative measurable function. Then
\[
    \int_M f = \int_0^{\infty} \ell_d(\{f > t\})\ dt.
\]
\end{thm}
\begin{proof}
    See \cite[Proposition I.3.3]{chavel_isoperimetric_2001}
\end{proof}
Equipped with these results, we now proceed to the
proof:
\begin{thm}
    \label{thm:dynplaptoone1}
    Let $M\subset\mathbb{R}^d$ be a compact $d$-dimensional 
    submanifold with Lipschitz boundary, $T: M\rightarrow M$ a smooth volume-preserving diffeomorphism and $p\geq 1$.
    Then
    \begin{equation}
        \dynamic{\lambda}_p(M, T) \geq
        \left(\frac{\dynamic{h}(M, T) }{p}\right)^p.
    \end{equation}
    \label{thm:cheegerdynamic}
\end{thm}

\begin{proof}
We modify the proof that is given in
    the appendix of \cite{lefton_numerical_1997} and
    \cite[Theorem 3]{kawohl_isoperimetric_2003}.
    
\textbf{Step 1.} First we prove the inequality for~$p=1$.
For this, assume $w\in C^{\infty}_0(M)$
    and define
    $A(t):= \{ x\in M \mid w(x) > t\}$.
    By the coarea formula
    \begin{equation}
        \int_M |\nabla w| =
        \int_{-\infty}^\infty \ell_{d-1}(\partial A(t))\ dt,
        \label{eq:coarea1}
    \end{equation}
    where the boundary is relative to $M$.
    If we apply the coarea formula to $T_*w$ we get
    \begin{align}
        \int_M |\nabla T_*w| &=
        \int_{-\infty}^\infty \ell_{d-1}(\partial \{x\in M\mid T_*w(x) > t\})\ dt.\\
        &=\int_{-\infty}^\infty \ell_{d-1}(\partial \{x \in M\mid w(T^{-1}(x)) > t\})\ dt.\\
        &=\int_{-\infty}^\infty \ell_{d-1}(\partial \{Ty \mid w(y) > t\})\ dt.\\
        &=\int_{-\infty}^\infty \ell_{d-1}(\partial \{Ty \mid y\in A(t)\})\ dt.\\
        &=
        \int_{-\infty}^\infty \ell_{d-1}(\partial(TA(t)))\ dt.
       \label{eq:coarea2}
    \end{align}
    Now by \Cref{lem:introducingabs}(c)
    we can express the integrals on the right-hand side of
    \eqref{eq:coarea1} and \eqref{eq:coarea2} by using the
    sets $B(t) := \{ x\in M \mid |w(x)| > t \}$ to get 
   \begin{align} 
        \int_M|\nabla w|&=
        \int_0^\infty \ell_{d-1}(\partial B(t))\ dt.\label{eq:coarea3}\\
        \int_M|\nabla T_*w|&=
        \int_0^\infty \ell_{d-1}(\partial(TB(t)))\ dt.\label{eq:coarea4}
    \end{align}
    Combining \eqref{eq:coarea3} and \eqref{eq:coarea4} we can continue with
    \begin{align}
        \int_{M} \frac12(|\nabla w|+|\nabla T_*w|)
        &= \frac12\int_0^\infty \ell_{d-1}(\partial B(t)) +
                                  \ell_{d-1}(\partial TB(t))\ dt\\
        &= \frac12\int_0^{\max |w|}
                             \frac{\ell_{d-1}(\partial B(t)) +
                                   \ell_{d-1}(\partial TB(t))}
                                  {\ell_d(B(t))}\ell_d(B(t))dt
                                  \label{eq:ratiobeforecheeger}
                                  \\
        &\geq \dynamic{h}(M, T) \int_0^{\max |w|} \ell_d(B(t))\ dt
                                  \label{eq:ratioaftercheeger}
        \\
        &=\dynamic{h}(M, T)\int_M |w|,\label{coareacavalieri}
    \end{align}
    where in the last step we have applied Cavalieri's principle to~$|w|$.

    For the step \eqref{eq:ratiobeforecheeger} to \eqref{eq:ratioaftercheeger}
    we also need to justify that $B(t)$ is in the admissible class
    for the Cheeger problem for almost all $t$. First observe
    that, as $|w|$ is continuous, and vanishes on the
    boundary of $M$, $B(t)$ does not touch $\partial M$ if
    $t>0$.
    As for the regularity of $\partial B(t)$, note that
    by \Cref{lem:introducingabs}(b) we have $\partial B(t)$ is the disjoint union of $\partial
    A(t)$ and $\partial A(-t)$ for almost all $t$.
     Sard's
    theorem tells us that $w^{-1}(t)$ contains no
    critical points for almost all $t$. Hence, by the implicit
    function theorem, $A(t)$ has $C^\infty$ boundary for
    almost all $t$ and so does $B(t)$.
    After applying the same argument
    to $TB(t)$ we know that we may indeed bound the
    fraction in \eqref{eq:ratiobeforecheeger} by~$\dynamic{h}(M, T)$.

    The inequality (\ref{coareacavalieri}) also holds for
    $w\in W_0^{1,1}(M)$ as
    $C^{\infty}_0(M)$ is dense in $W^{1,1}_0(M)$. This shows the claim for~$p=1$.

\textbf{Step 2.} Next we generalise to~$p>1$, essentially by reducing the result to the case in Step~1.
    For $p>1$ and $v\in W^{1,p}(M)$ define
    $\Phi(v):= |v|^{p-1}v$. As $(|x|^{p-1}x)' = p|x|^{p-1}$, we 
    can use the chain rule to get that the 
    gradient of $\Phi(v)$ is 
    \[
        \nabla \Phi(v(x)) =
        p|v(x)|^{p-1}\nabla v(x). 
    \]
    By Hölder's inequality we can estimate
    \begin{align}
        \int_{M} |\nabla \Phi(v)| =
        p\int_M |v|^{p-1}|\nabla v| \leq
        p\left(\int_M
        |v|^{(p-1)\frac{p}{p-1}}\right)^{\frac{p-1}{p}}
        \|\nabla v\|_p = p\|v\|_p^{p-1}\|\nabla v\|_p
        \label{holder},
    \end{align}
    implying that $w:=\Phi(v)\in W^{1,1}_0(M)$
    for~$v\in W^{1,p}_0(M)$. 
  
    We repeat for $T_*v$. 
    First note that 
    \[
        \Phi(T_*v(x)) = \Phi(v(T^{-1}(x))) = |v(T^{-1}x)|^{p-1}v(T^{-1}x) = \left[T_*(|v|^{p-1}v)\right](x) =
        T_*(\Phi(v))(x)
    \]
    and hence~$\Phi(T_*v) = T_*(\Phi(v))$.
    Thus, as above: 
    \begin{equation}
        \int_{M} |\nabla T_*w|=
        \int_{M} |\nabla T_* \Phi(v)|=
        \int_{M} |\nabla \Phi(T_*(v))|
        \leq p \|T_*v\|^{p-1}_p\|\nabla T_*v\|_p
        =p \|v\|^{p-1}_p\|\nabla T_*v\|_p
        \label{holder2}
    \end{equation}
    using that $T$ is volume-preserving in the last step.
    Now on the right-hand side of (\ref{holder2}), both norms are bounded by
    compactness of $M$ and smoothness of $T$. Hence,
    also~$T_*w\in W^{1,1}_0(M)$.
    We can therefore apply inequality \eqref{coareacavalieri}
    to get
    \begin{align}
        \dynamic{h}(M, T)
        &\overset{\eqref{coareacavalieri}}{\leq}
        \frac{\int_{M} |\nabla w| + |\nabla T_*w| }{2\int_M|w|}\\
        &\overset{\eqref{holder}, \eqref{holder2}}{\leq}
        \frac{p\|v\|^{p-1}_p(\|\nabla v\|_p + \|\nabla T_*v\|_p)}{2\int_{M} |v|^p}\\
        &= p\frac{\|\nabla v\|_p + \|\nabla T_*v\|_p}{2\|v\|_p}
    \end{align}
    Raising this to the power of $p$ and using that convexity of $x^p$ implies
    $(\frac{a+b}{2})^p \leq \frac{a^p+b^p}{2}$ we arrive at
    \begin{align}
        \left(\frac{\dynamic{h}(M, T)}{p}\right)^p
        &\leq \left(\frac{\|\nabla v\|_p + \|\nabla T_*v\|}{2\|v\|_p}\right)^p\\
        &\leq \frac{\|\nabla v\|_p^p + \|\nabla T_*v\|_p^p}{2\|v\|^p_p}
    \end{align}
    which gives the desired inequality after passing to the infimum.
\end{proof}

Finally, we can turn our attention to the behaviour of $\dynamic{\lambda}_p$ as~\mbox{$p\to 1$}.
\begin{thm}
    Let $M\subseteq \mathbb{R}^d$ and $T:M\rightarrow M$ be as in \Cref{thm:dynplaptoone1}. Then the first eigenvalue $\dynamic{\lambda}_p(M, T)$ converges to
    $\dynamic{h}(M, T)$ as~$p\rightarrow 1$. 
    \label{thm:cheegerconv}
\end{thm}
\begin{proof}
    We use a mixture of arguments from \cite[Corollary 6]{kawohl_isoperimetric_2003},
    and \cite[Lemma A.1 and Theorem 3.1]{froyland_dynamic_2015}.
    By \Cref{thm:cheegerdynamic} we already have a lower
    bound on $\dynamic{\lambda}_p(M, T)$ that converges to
    $\dynamic{h}(M, T)$ as~$p\rightarrow 1$. To
    construct an upper bound, consider a subdomain
    $\subsetname_k\subset M$ with smooth boundary not touching
    $\partial M$
    and
    \[
    \frac{\ell_{d-1}(\partial \subsetname_k)+\ell_{d-1}(\partial T \subsetname_k)}{2\ell_d(\subsetname_k)} -
    \dynamic{h}(M, T) <\frac1k \quad\text{for a $k\in\mathbb{N}$}.
    \]
    Define
    \begin{align}
        f_\varepsilon(x) :=
        \begin{cases}
            1 & x\in \subsetname_k\\
            1-\frac{1}{\varepsilon} d(x, \subsetname_k) &
            0<d(x, \subsetname_k) < \varepsilon\\
            0 & d(x, \subsetname_k) \geq \varepsilon\\
        \end{cases}
    \end{align}
    with $d(x, \subsetname_k):=\inf_{y\in \subsetname_k}|x-y|$.
    Then $f_{\varepsilon}$ is in $W_0^{1,p}(M)$ for $\varepsilon$ small enough and 
    \begin{equation}
        \dynamic{\lambda_p}(M, T) \leq
        \frac{\|\nabla f_{\varepsilon}\|_p^p +
              \|\nabla T_*f_{\varepsilon}\|^p_p}
              {2\|f_{\varepsilon}\|_p^p}=:
              C_{p, \varepsilon, k}.
    \label{eq:upperboundonlam}
    \end{equation}
    To remove the dependence on $p$, first
    note that $|\nabla f_{\varepsilon}|\leq G_\varepsilon$ for some
    constant $G_\varepsilon$ and $|\nabla f_{\varepsilon}|^p$ converges pointwise
    to $|\nabla f_{\varepsilon}|$ for $p\rightarrow 1$,
    hence by compactness of $M$ and dominated convergence~$ \lim_{p\to 1}\|\nabla f_{\varepsilon}\|_p^p =
    \|\nabla f_{\varepsilon}\|_1$.  Using that $T$ is continuously
    differentiable and applying the same argument to $\|\nabla T_*f_{\varepsilon}\|_p^p$ and $\|f_{\varepsilon}\|_p^p$ yields
    \[
        \lim_{p\rightarrow1} C_{p,\varepsilon, k} =
         \frac{\|\nabla f_{\varepsilon}\|_1 +
              \|\nabla T_*f_{\varepsilon}\|_1}
              {2\|f_{\varepsilon}\|_1}=:
              C_{\varepsilon, k}.
    \]
    
    Now $\lim_{\varepsilon\rightarrow 0}\|f_\varepsilon\|_1
    =\ell_{d}(\subsetname_k)$ by dominated convergence and
    $\lim_{\varepsilon\rightarrow0}\|\nabla f_{\varepsilon}\|_1=
    \ell_{d-1}(\partial \subsetname_k)$, as $|\nabla f_\varepsilon|=\frac1\varepsilon$ almost everywhere on $\Gamma_\varepsilon:=\{x\in M\backslash \subsetname\mid d(x,M)\leq \varepsilon\}$ and zero elsewhere so $\lim_{\varepsilon\rightarrow 0}\int_M|\nabla f_\varepsilon|=\lim_{\varepsilon \rightarrow 0}\ell_d(\Gamma_\varepsilon)/\varepsilon = \ell_{d-1}(\partial \subsetname)$ (see also proof of~\cite[Theorem 3.1]{froyland_dynamic_2015}). 
    Following the proof of \cite[Theorem 3.1]{froyland_dynamic_2015},
    using \cite[Lemma A.1]{froyland_dynamic_2015} we have
    $
        \lim_{\varepsilon\rightarrow 0}
        \|\nabla T_* f_{\varepsilon}\|_1 =
        \ell_{d-1}(T \partial \subsetname_k)
        = \ell_{d-1}(\partial T\subsetname_k)
    $,
    which leads us to 
    \[
        \lim_{\varepsilon\rightarrow0} C_{\varepsilon, k}  =
        \frac{\ell_{d-1}(\partial \subsetname_k) + \ell_{d-1}(\partial T\subsetname_k)}
        {2\ell_{d}(\subsetname_k)}
       =: C_k.
    \]
    Finally by definition of $\subsetname_k$ we have
    \begin{equation}
        \lim_{k\rightarrow \infty} C_k = \dynamic{h}(M, T).
        \label{eq:lim-of-ck}
    \end{equation}
    Now we can conclude by taking successive
    limits.
    By \eqref{eq:upperboundonlam}
    we have for fixed $k$ and~$\varepsilon$
    \[
        \limsup_{p\rightarrow 1}
        \dynamic{\lambda}_p(M, T) \leq C_{\varepsilon, k},
    \]
    which, after passing to the limit $\varepsilon \rightarrow 0$
    implies that for all~$k$
    \[
        \limsup_{p\rightarrow 1}
        \dynamic{\lambda}_p(M, T) \leq C_{k}.
    \]
    A final limit $k\rightarrow\infty$ together with
    \Cref{eq:lim-of-ck} shows
    \[
         \limsup_{p\rightarrow 1}
        \dynamic{\lambda}_p(M, T) \leq \dynamic{h}(M, T).
    \] Recalling that \Cref{thm:cheegerdynamic} gave us
    $
        \liminf_{p\rightarrow 1} \lambda_p(M, T) \geq
        \dynamic{h}(M, T)
    $ we can conclude that
    \[
       \lim_{p\rightarrow 1} \dynamic{\lambda}_p(M, T) =
       \dynamic{h}(M, T).
   \]
\end{proof}

\section{Numerical approximation of eigenfunctions}

Finding eigenfunctions of the $p$-Laplacian computationally is much harder for $p\neq 2$ than for~$p=2$.
In the latter case, the operator is the standard linear Laplacian and finite element methods can readily be applied~\cite{StrangFix:73}.  For the $p$-Laplacian,  Hor\'ak \cite{horak_numerical_2011} uses a modified gradient descent for the first eigenfunction and a mountain pass algorithm for the second one. Lefton and Wei~\cite{lefton_numerical_1997} use a penalty formulation together with the Levenberg--Marquard algorithm to approximate the first eigenfunction. 
For our purpose, we will modify a method by Yao and Zhou \cite{yao_numerical_2007}, as it is general enough
to be easily applied to our dynamical functional $\dynamic{F}$.
    
\subsection{Abstract algorithm for the first eigenpair}

We will use the algorithm mostly for the first eigenfunction. 
As it simplifies substantially in this case and is essentially a 
descent algorithm, we summarise this special case here. Recall that for the first eigenpair $(u_p,\lambda_p)$ of the $p$-Laplacian, we have 
\[
    \lambda_p = \inf_{u\in W^{1,p}_0(M)\backslash\{0\}} J(u)
\]
and the infimum is attained at the first eigenfunction $u_p$. 
For the descent algorithm, \cite{yao_numerical_2007} proposes to use the descent direction $-d\in W^{1,q}_0(M)$ that fulfils 
\begin{equation}
\label{eq:weakpseudogradeq1}
    \int_M \nabla d \cdot\nabla v = J'(u)(v)
\end{equation}
for all $v\in W^{1,p}_0(M)$. We will denote this descent direction
by $\grad J(u):=d$. In the case of 
$C^1$ boundary, it is shown in \cite[Theorem 7.5]{simader1972dirichlet}  that such a~$d$ exists.\footnote{To apply \cite[Thm. 7.5]{simader1972dirichlet}, set $B[\mathbf{u},\Phi]:=\int_M \nabla \mathbf{u}\nabla \Phi$, $F(\Phi) := J'(u)(\Phi)$ 
(using the respective notations $B, \mathbf{u}$ and $\Phi$ from \cite{simader1972dirichlet} and $u, F, J'$ from our definitions) and note that the set $N_q$ from the assumptions of \cite[Thm. 7.5.]{simader1972dirichlet} equals $\{0\}$, hence trivially $F(z)=0$ for all $z\in N_q$.} 
Note that because $M$ is bounded and $p<q$ we have $W^{1,q}_0(M)\subseteq W^{1,2}_0(M)\subseteq W^{1,p}_0(M)$, so we also get $d\in W^{1,p}_0(M)$. 
Note further that $J'(u)(d) = \|\nabla d\|_2^2$ and thus 
\begin{align}
   J(u-td)  &= J(u) - t\, J'(u) (d) + o(t) \\
                                & = J(u) - t\|\nabla d\|_2^2 + o(t) \\
                                & < J(u) - t\|\nabla d\|_2^2/4\label{eq:descentdirection}
 \end{align}
for small $t$, i.e.\ $-d$ is a descent direction and a suitable stepsize $t$ can be found by an Armijo-type stepsize search. 

The corresponding descent scheme is given in \Cref{alg:YaoZhouFirst}, which 
is based on the local min-max algorithm from \cite[Section 3]{yao_numerical_2007}, specialised for the first eigenfunction and with the modifications
from \cite[Section 4.1]{yao_numerical_2007} filled in.
\begin{algorithm}[H]
\caption{Computation of the first eigenfunction}\label{alg:YaoZhouFirst}
\begin{enumerate}
    \item Set $k=1$, choose an initial guess $u^{[1]}\in W^{1,p}_0(M)$ and a parameter $\alpha>0$.  
    \item While $\|\grad J(u^{[k]})\|_p$ is too large do    
    \begin{enumerate}
        \item Compute the descent direction $d^{[k]} := -\grad J(u^{[k]})$.
        \item For $s\in \mathbb{R}$ define
        \[
            u^{[k+1]}(s) :=\frac{u^{[k]} + sd^{[k]}}{\|u^{[k]} + sd^{[k]}\|_p}.
        \]
        \item Set $s^*:=\alpha/\max(1,\|d^{[k]}\|_p)$
            and check whether
        \[
            J(u^{[k+1]}(s^*)) \leq J(u^{[k]}) -\frac{s^*}{4}\|\nabla d^{[k]}\|_2^2
        \] 
        is fulfilled. If it is not, keep halving $s^*$ until it is.
        \item 
        Set $u^{[k+1]} := u^{[k+1]}(s^*)$.
        \item Increment $k\to k+1$. 
        \item If $s^*$ is smaller than some tolerance, go to Step 3. 
    \end{enumerate}
    \item Return $u_p:=u^{[k]}$.
\end{enumerate}
\end{algorithm}

\subsection{Application to the static and the dynamic $p$-Laplacian}
\label{sec:concrete_choices}

\paragraph{The functionals.}

For the static $p$-Laplacian $\Delta_p$ we have $F(u)=\|\nabla u\|_p^p$
and $G(u)=\|u\|_p^p$ so that the derivative of $J$ is
\begin{align}
    J'(u)(v) &= \frac{1}{G(u)^2}(G(u)F'(u)(v) - F(u)G'(u)(v))\\&=\frac{p}{b^2}\int_M(b|\nabla u|^{p-2}\nabla u \nabla v - a |u|^{p-2} uv),
\end{align}
where $a=\|\nabla u\|_p^p$ and $b=\|u\|_p^p$ \cite[Section 4.1]{yao_numerical_2007}. Hence, the descent direction $d$ is given by the solution of
\[
    \int_M\nabla d \cdot\nabla v = 
    \frac{p}{b^2}\int_M(b|\nabla u|^{p-2}\nabla u \nabla v
    - a u|u|^{p-2}v) \qquad \text{for all } v\in W_0^{1,p}(M).
\]
In case of the dynamic $p$-Laplacian $\dynamic{\Delta}_p$, we consider 
\[
    \dynamic{J}(u) := \frac{\dynamic{F}(u)}{G(u)} = \frac12\frac{F(u) + F(T_*u)}{G(u)}. 
\]
As in \eqref{eq:derivativeFbar} we have
\[
    (\dynamic{F})'(u)v = \frac12(F'(u)v + F'(T_*u)T_* v),
\]
so the derivative of $\dynamic{J}$ is 
\begin{align}
    (\dynamic{J})'(u)v &= \frac{1}{G(u)^2}(G(u)(\dynamic{F})'(u)v - 
    \dynamic{F}(u)G'(u)v) \\ 
    &= \frac{1}{b^2}\left(\frac12 b (F(u)v + F(T_*u)T_*v) 
            - \dynamic{a} G'(u)v\right)\\
    &=\frac{p}{b^2}\int_M\left(\frac12\left(|\nabla u|^{p-2}\nabla u \cdot \nabla v + |\nabla (u\circ T)|^{p-2}\nabla(u\circ T)\cdot\nabla(v\circ T)\right)- \dynamic{a} |u|^{p-2}u v \right), 
\end{align}
where $\dynamic{a}:= \frac1p\dynamic{F}(u)$. Defining $A(x):=DT^{-T}(x)$ and using volume-preservation of $T$ we make the replacement
\[
\int_M |\nabla (u\circ T)|^{p-2}\nabla(u\circ T)\cdot\nabla(v\circ T) = \int_M|A\nabla u|^{p-2}(A\nabla u)\cdot (A\nabla v).
 \]
 This yields 
 \[
    (\dynamic{J})'(u)v = 
    \frac{p}{b^2}\int_M\left(\frac12\left(|\nabla u|^{p-2}\nabla u \cdot \nabla v + |A\nabla u|^{p-2}(A\nabla u)\cdot (A\nabla v)\right)- \dynamic{a} |u|^{p-2}u v \right). 
\]

\paragraph{The descent direction.}

We found that in the case of the dynamic $p$-Laplacian, it is not advisable
to define $d$ using \eqref{eq:weakpseudogradeq1} directly replacing $J'$ with $(\dynamic{J})'$, as the resulting $d$ exhibits high irregularity in areas where the norm of the derivative of $T$ is large, leading to undesirably small steps in the descent algorithm, similar to the problems gradient descent exhibits when not conditioned properly. 

As a remedy, we instead use $\dynamic{d}$ defined by 
\begin{equation}
\label{eq:definitionbard}
    \int_M \nabla \dynamic{d}\cdot\left(\frac12(I+A^TA)\nabla v \right)= (\dynamic{J})'(u)v \qquad \text{for all } v\in W^{1,p}_0(M)
    \end{equation}
    and denote this choice by $\dynamic{\grad} \dynamic{J}:=\dynamic{d}$. Note that
    because 
    \[
         \Big\langle \ell, \frac12(I+A(x)^TA(x))\ell\Big\rangle \geq \frac12|\ell|^2
    \]
    for all $x\in M$ and $\ell\in\mathbb{R}^d$, and the coefficients of $A$ are bounded in $M$, the left-hand side of \eqref{eq:definitionbard} defines a uniformly strongly elliptic regular Dirichlet bilinear form of order $1$, so \cite[Theorem~7.5]{simader1972dirichlet} is again applicable if $M$ has $C^1$ boundary.
    
    In our experiments, this choice of $\dynamic{d}$ lead to higher regularity of the descent direction, resulting in considerably bigger step sizes 
(see \Cref{fig:pseudograds} for a visual comparison of $\grad \dynamic{J}$ and $\dynamic{\grad}\dynamic{J}$).  

\begin{figure}[ht]
    \centering
    \parbox{0.4\textwidth}
    {\includegraphics[width=0.4\textwidth]
    {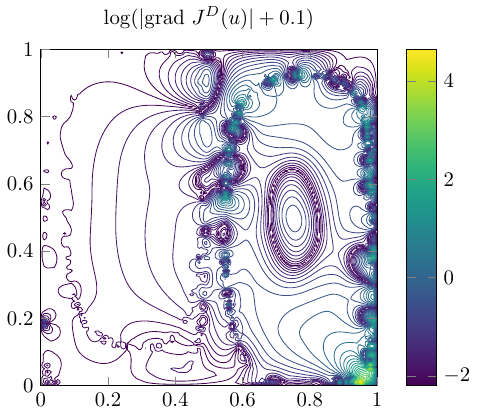}}
    \qquad
    \parbox{0.4\textwidth}
    {\includegraphics[width=0.4\textwidth]
    {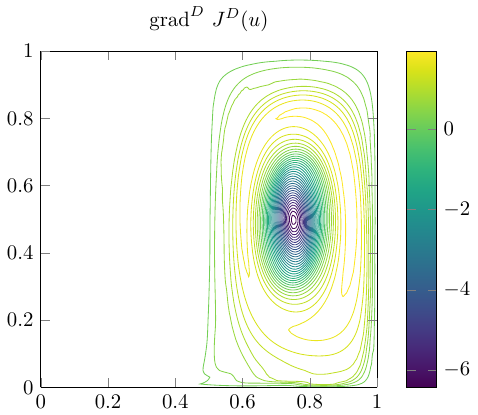}}
    \caption{The gradients $\grad \dynamic{J}(u)$ (left) and $\dynamic\grad \dynamic{J}(u)$  (right) at the starting point $u$ of \Cref{alg:YaoZhouFirst} for $p = 1.83$ in the example in \cref{sec:exp_rot_gyre}. As a starting point $u$ we choose the first eigenfunction of $\dynamic\Delta_2$.  Note the log transformation in the left plot, which was used in order to visualise the different orders of magnitude.}
    \label{fig:pseudograds}
\end{figure}

\paragraph{Initialization.}

Finally, one needs to find an initial guess for $u^{[1]}$  in the 
first step of \Cref{alg:YaoZhouFirst}. Yao and Zhou suggest 
finding functions with appropriate nodal line structures (see 
\cite[Remark 3.1(a)]{yao_numerical_2007}). We used the first eigenfunction of $\dynamic{\Delta}_2$ for this purpose, which we computed using the Cauchy--Green approach from \cite[Section 3.1]{froyland_robust_2018}.

\subsection{The second eigenpair}

One of our numerical examples in \cref{sec:experiments} is on the torus $\mathbb{T}^2$, which has no boundary. This has the consequence that the infimum of $\dynamic{J}$ is~$0$, realised by any constant function and carrying no information about the dynamics.  We thus need to compute the second eigenfunction and therefore briefly describe the algorithm of Yao and Zhou for this case and present our version for the dynamic $p$-Laplacian~$\smash{ \dynamic{\Delta}_p }$.

Note that the first eigenvalue $\lambda_p=\lambda_p^{(1)}$ of $\Delta_p$ is isolated \cite[Theorem 9]{lindqvist_nonlinear_2008}. As  the spectrum of $\Delta_p$ is closed~\cite[Theorem 3]{lindqvist_nonlinear_2008}, 
we may speak of the second eigenvalue $\lambda^{(2)}_p$ of $\Delta_p$ 
as the infimum over all eigenvalues bigger than~$\lambda_p^{(1)}$. We assume this to be true for the dynamic $p$-Laplacian $\smash{\dynamic{\Delta}_p}$ as well.
The second eigenvalue can be characterised by the min-max-principle \cite[Equation (2.1)] {yao_numerical_2007}(cf.\ Appendix~\ref{ssec:minmax} for a derivation in the linear case) 
\begin{equation}
\label{eq:minmax}
    \lambda^{(2)}_p = \min_{v\in L'}\max_{u\in [L, v] \cap S} J(u),   
\end{equation}
where $L$ is the span of the first eigenfunction $u_p=u_p^{(1)}$, $L'$ is a complement of $L$ (i.e.\ such that $W^{1,p}_0(M)=L\oplus L'$) and $[L,v]:=\spn(L, v)$.
By $S:= \{u\in W^{1,p}_0(M)\mid \|u\| = 1\}$, we denote the unit sphere in $W^{1,p}_0(M)$ and define $[L,v]_S:= \{t_1u_p^{(1)}+t_2v\mid t_1^2+t_2^2 = 1\}$.
Since $J(tu)=u$ for all $t\neq 0$, and $[L,v]_S \xrightarrow{{u\ \mapsto\ u/\|u\|}} [L,v]\cap S$ is a bijection, we have 
\begin{equation}
    \label{eq:equivalence}
\max\{J(u)\mid u\in [L,v]_S\}=\max\{J(u)\mid u\in [L,v]\cap S\}.
\end{equation}
We work directly with the left-hand side of (\ref{eq:equivalence}) as the optimisation is now finite-dimensional over $t_1,t_2$.
From a maximiser of the left-hand side of (\ref{eq:equivalence}) we can immediately obtain a maximiser of the right-hand side.

A local optimum of the maximisation in \eqref{eq:minmax} can then be found with a finite-dimensional optimisation over the~$t_i$ (see \cref{sec:exp_standard_map} for implementation details). 
If $u^+:L'\rightarrow W^{1,p}_0(M)$ is a \emph{peak selection} of $J$, meaning
\[
u^+(v) \in \arg\max\{J(u)\mid u\in [L,v]_S\},
\]
it is thus sufficient to find
\begin{align}
\label{eq:optimizationproblem_second}
    v_2:=\arg\min \{J(u^+(v)) \mid v\in L'\}
    \end{align}
to get a second eigenfunction $u_p^{(2)}:= u^{+}(v_2)$, i.e.\ $J(u_p^{(2)}) = \lambda_p^{(2)}$.  The minimisation problem \eqref{eq:optimizationproblem_second} can be solved with a similar descent scheme as \Cref{alg:YaoZhouFirst}:  The crucial observation is that
$-d$ can be used as a descent direction for $v\mapsto J(u^+(v))$ \cite[Lemma 2.5]{yao_numerical_2007}. Note that this is subtly different from the 
obvious application of \eqref{eq:descentdirection}
    which would be that ${\frac{d}{dt}J(u^+(v) -td)|_{t=0}<0}$, as opposed to ${\frac{d}{dt}J(u^+(v-td))|_{t=0}<0}$.
The corresponding scheme is given in \Cref{alg:yaozhousecond}. 

\begin{algorithm}[H]
\caption{Computation of a second eigenfunction}
\label{alg:yaozhousecond}
\begin{enumerate}
    \item Compute the first eigenfunction $u^{(1)}_{p}$ with \Cref{alg:YaoZhouFirst}.
    \item Set $k=1$, choose an initial guess 
    $v^{[1]}\in W^{1,p}_0(M)$ and a parameter $\alpha>0$. 
    \item While $\big \|\grad J(u^+(v^{[k]})) \big\|_p$ is too large do
    \begin{enumerate}
        \item Compute a descent direction $d^{[k]} := -{\rm sign}(t_2(v^{[k]}))\grad J(u^+(v^{[k]}))$.
        \label{alg:yaozhousecond3a}
    \item For $s\in\mathbb{R}$ define
        \[
            v^{[k+1]}(s) :=\frac{v^{[k]} + sd^{[k]}}{\|v^{[k]} + sd^{[k]}\|_p}.
        \]
        \item Set $s^*:=\alpha/\max(1,\|d^{[k]}\|_p)$ and check whether
            \[
                J(u^+(v^{[k+1]}(s^*))) \leq  J(u^+(v^{[k]})) -\frac{s^*}{4}|t_2(v^{[k]})|\|\nabla d^{[k]}\|_2^2
            \]
        is fulfilled. If it is not, keep halving $s^*$ until it is.
        \item 
        Set $v^{[k+1]} := v^{[k+1]}(s^*)$.
        \item Increment $k\to k+1$.
        \item If $s^*$ is smaller than some tolerance, go to Step 4. 
    \end{enumerate}
    \item Return $u_p^{(2)}:=u^+(v^{[k]})$.
\end{enumerate}
\end{algorithm}

\section{Experiments}
\label{sec:experiments}

We perform a series of experiments to illustrate  how the eigenfunctions of the dynamic $p$-Laplacian -- and in particular their level sets with optimal Cheeger ratio -- compare to those of the $2$-Laplacian.
We choose a variety of domains:  fully periodic, partially periodic, and non-periodic, and a variety of dynamics with differing levels of symmetry.
We find that the Cheeger ratios of many superlevel sets of the first eigenfunction improve considerably for $p\rightarrow 1$, indicating that the dynamic 
$p$-Laplacian behaves similarly to \Cref{thm:plaptoone}(\ref{thm:plaptoone_b}). 
We also find that the level set of the leading eigenfunction of the $p$-Laplacian with \emph{least Cheeger ratio} is relatively insensitive to $p$ for $1.3\le p\le 2$;  this finding is replicated for the dynamic $p$-Laplacian.
The code for replicating our results can be found in the accompanying package \href{https://github.com/adediego/DynamicPLaplacian.jl}{\texttt{DynamicPLaplacian.jl}}.

\subsection{The static unit square}
\label{sec:exp_static}

We first look at $M = [0,1]^2$ and $T=\mathrm{id}$.
We get the functional corresponding to the classical Cheeger
problem
\[
    J(u) = \frac{\|\nabla u\|_p^p}{\|u\|_p^p}.
\]
The theory predicts that for $p\rightarrow 1$ the first eigenvalue $\lambda_p$ converges to the Cheeger constant
\[
    h(M) = \inf_{\subsetname\subset M}\frac{\ell_{1}(\partial \subsetname)}{\ell_2(\subsetname)} = \frac{4-\pi}{2-\sqrt\pi} \approx3.772,
\]
(see \cref{sec:effect_of_relaxation}), while the Cheeger set is a square with rounded corners of radius~$1/h(M)\approx 0.265$. 

We approximate the first eigenfunction $u_p$ of $-\Delta_p$ with
Dirichlet boundary conditions and investigate how the Cheeger ratio of its level sets changes with $p$. For the discretisation of the eigenproblem \eqref{eq:dynamic_eigenproblem} we use a finite element method with linear triangular Lagrange elements on a $100\times100$ grid, implemented using the package \texttt{Gridap.jl} \cite{Badia2020}. For the stopping criterion in \Cref{alg:yaozhousecond},
we use
$\|\dynamic{\grad}\ \dynamic{J}\|_p \leq 10^{-3}$.
Note that in this example there is no
dynamics, i.e.\ we set $T:=id$, and thus $\dynamic\grad \dynamic{J}=\grad J$.
While for $p$ close to $2$ this is generally achieved in  a few iterations, the  number of iterations grows quickly as $p$ gets close to 
$1$; see Figure~\ref{fig:iterations}.  We restrict the presentation of results
to the range of $p$ where the convergence criterion
was achieved, namely $1.3\leq p\leq 2.0$.
\begin{figure}[ht]
    \center
    \includegraphics[height=8cm]{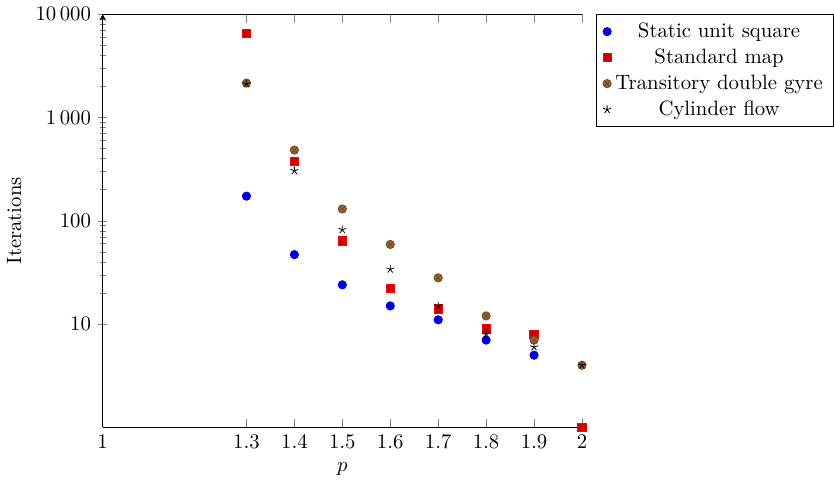}
    \caption{
    Number of iterations needed to achieve the convergence
    criterion ${\|\grad J\|_p \leq 10^{-3}}$ 
    (or ${\|\dynamic\grad \dynamic{J}\|_p \leq 10^{-3}}$ 
    respectively).
    For details see sections \ref{sec:exp_static} (static unit square),
    \ref{sec:exp_rot_gyre} (transitory double gyre), 
    \ref{sec:exp_standard_map} (standard map) 
    and \ref{sec:exp_cylinder} (cylinder flow). 
    }
    \label{fig:iterations}
\end{figure}

In Figure \ref{fig:static_unit_square_efuns}, we plot the leading eigenfunction of $\Delta_p$ for three values of $p$ (upper row) and their level sets (lower row). The level sets were obtained by the marching squares algorithm implemented in \texttt{Contour.jl} \cite{Contour}.  The first eigenfunction $u_p$ does indeed get \enquote{flatter} with smaller $p$, and the level sets get closer to each other. 

\begin{figure}
\begin{tabular}{ c c c }
    \includegraphics
    [width=5cm, trim ={0 0 1.5cm 0}, clip] {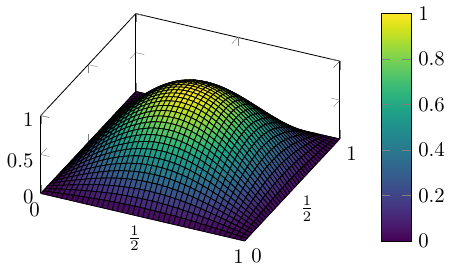} &
    \includegraphics
    [width=5cm, trim ={0 0 1.5cm 0}, clip] {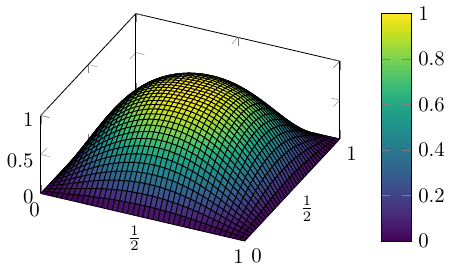} &
    \includegraphics[width=5cm, trim ={0 0 1.5cm 0}, ]{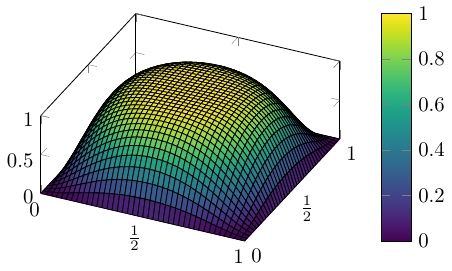}\\
    \includegraphics[width=4cm]{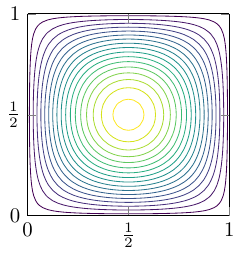} &
    \includegraphics[width=4cm]{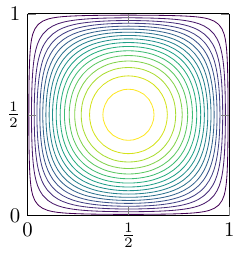} &
    \includegraphics[width=4cm]{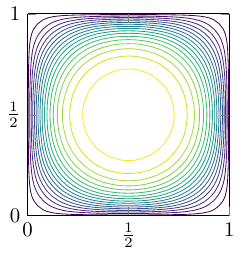}
\end{tabular}
\caption{The first eigenfunction of $\Delta_p$  
for  $p=2.0$ (left), $p=1.6$ (centre) and $p=1.3$ (right)
on  the unit square.}
\label{fig:static_unit_square_efuns}
\end{figure}

We next compute the Cheeger ratio for specific level sets of the first eigenfunction.  For the computation of the area of the superlevel sets we used the well-known formula sometimes called the \emph{shoelace formula} (see e.g. \cite[3.5 (6.)]{rade_mathematics_2004}. If the approximation of some eigenfunction is negative somewhere, it can happen that a level set corresponding to a negative value has a very small Cheeger ratio. We thus only consider level sets corresponding to positive values in such cases. We observe in~\Cref{fig:static_unit_square_statistics} that the Cheeger ratio of every single level set improves with decreasing $p$, while the Cheeger ratios of the best level sets barely change with varying~$p$.
To quantify the improvement, we investigate the median of the Cheeger ratio of the level sets with respect to the uniform distribution on the range of the eigenfunction. 
\begin{figure}
    \centering
    \begin{tabular}{c c c}
        \includegraphics[height=4.5cm]{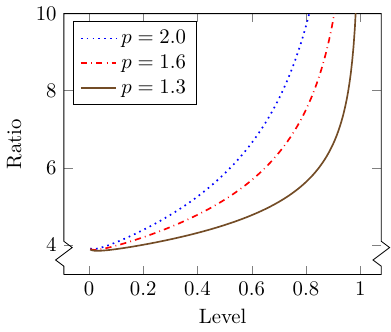}&
        \includegraphics[height=4.4cm]{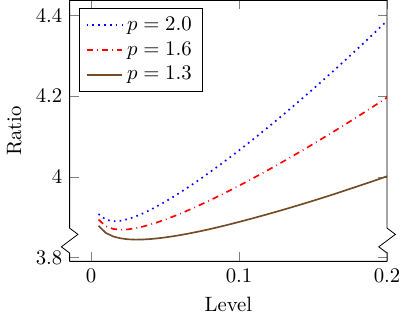}&
        \includegraphics[height=4.4cm]{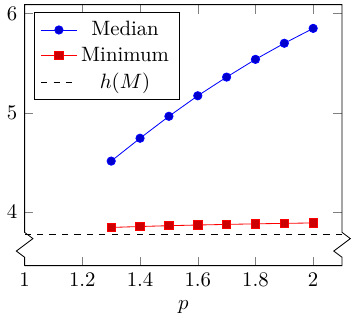}
    \end{tabular}
    \caption{Cheeger ratio of level sets of the first eigenfunction of $\Delta_p$ (left, closeup in the middle) on the unit square and statistics of the ratio when choosing the level set randomly (right). Note that the ratio does not depend monotonically on the level.
}
    \label{fig:static_unit_square_statistics}
\end{figure}
Finally, we check the best level set for different $p$ and compare
it to the boundary of the actual Cheeger set in \Cref{fig:static_unit_square_cheeger}.
\begin{figure}[ht]
    \centering
    \begin{tabular}{c c}
        \includegraphics[height=5cm]{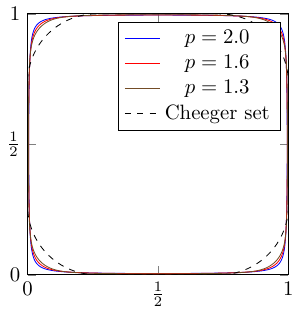}
        \includegraphics[height=5cm]{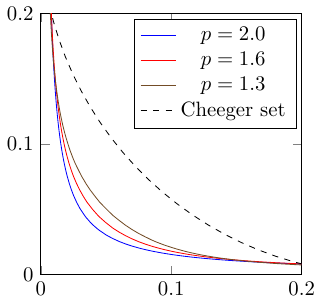}
    \end{tabular}
    \caption{The level set with the smallest  Cheeger ratio for different $p$ versus the actual Cheeger set (left) on the unit square. Closeup of the lower left corner of the square (right).}
    \label{fig:static_unit_square_cheeger}
\end{figure}

\subsection{The transitory double gyre}
\label{sec:exp_rot_gyre}

As a first example with dynamics we use the transitory double gyre introduced in~\cite{mosovsky_transport_2011}. We consider the time one flow map of the non-autonomous differential equation $(\dot x, \dot y) = (\partial_y \psi, -\partial_x \psi)$ on $M=[0,1]^2$ defined by the stream function
\begin{align*}
   \psi(x,y,t) = (1-s(t))\psi_P(x,y) + s(t)\psi_F(x,y) 
\end{align*}
with $\psi_P(x,y) = \sin(2\pi x) \sin(\pi y)$ and $\psi_F(x,y) = \sin(\pi x) \sin(2\pi y)$ and
\[
    s(t) = \begin{cases} 0 & \text{ for } t<0, \\
                         t^2(3-2t) & \text{ for } t\in [0,1], \\
                         1 & \text{ for } t > 1.
            \end{cases}
\]
In this flow, there are two vertically elongated gyres next to each other at $t=0$ and these smoothly transition via rotation by $90$ degrees to two horizontally elongated ones at ${t=1}$. We integrate the flow with the Tsitouras 5/4 Runge--Kutta method implemented in the package \texttt{DifferentialEquations.jl}  \cite{rackauckas2017differentialequations} using a tolerance of $10^{-7}$ and approximate the Cauchy--Green tensor with automatic differentiation, 
for which we use \texttt{ForwardDiff.jl}\  
 \cite{RevelsLubinPapamarkou2016}.

The dynamic $p$-Laplacian, the first eigenfunction $\dynamic{u_p}$, and its level sets are approximated as in the previous example. Note that the dynamic $p$-Laplacian acts on functions at the initial time $t=0$.
As in the static case, $u_p$ becomes flatter at its extremum for decreasing $p$; see~\Cref{fig:double_gyre_up}.  In Figure~\ref{fig:double_gyre_statistics}, individual level sets and their statistics are shown, as well as the optimal level sets for three different values of $p$. The overall shape of the optimal set does not vary significantly with $p$.

\begin{figure}[H]
\begin{tabular}{ c c c }
\includegraphics[width=5cm, trim={0 0 1.5cm 0}, clip]{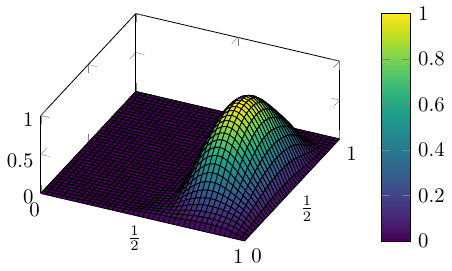} &
\includegraphics[width=5cm, trim={0 0 1.5cm 0}, clip]{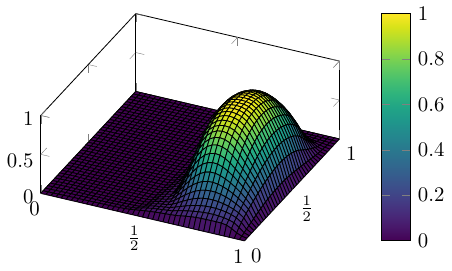} &
\includegraphics[width=5cm, trim={0 0 1.5cm 0}]{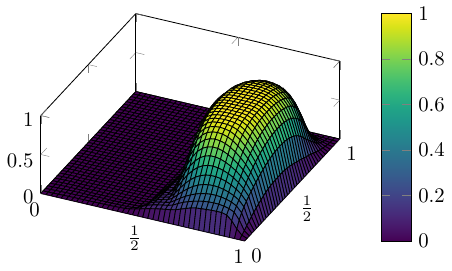}\\
\includegraphics[width=4cm]{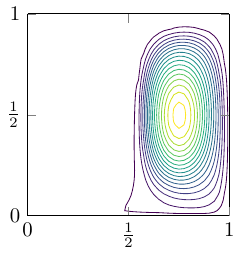} &
\includegraphics[width=4cm]{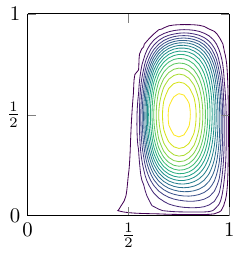} &
\includegraphics[width=4cm]{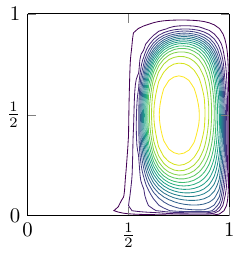}
\end{tabular}
\caption{The first eigenfunction of $\dynamic{\Delta}_p$ for the transitory
double gyre and
$p=2.0$ (left), $p=1.6$ (centre) and $p=1.3$ (right).}
\label{fig:double_gyre_up}
\end{figure}

\begin{figure}[H]
    \centering
    \includegraphics[width=0.33\textwidth,valign=t]{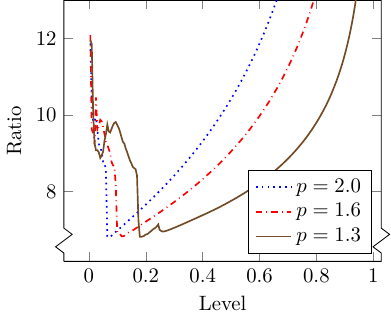}
    \hfil
    \includegraphics[width=0.31\textwidth,valign=t]{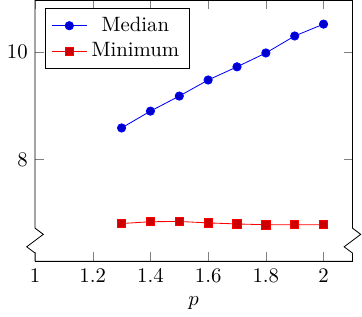}
    \\
    \includegraphics[width=0.27\textwidth,valign=t]{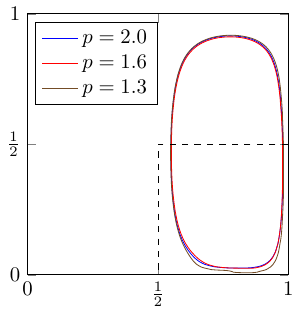}
    \hfil
    \includegraphics[width=0.27\textwidth, valign=t]{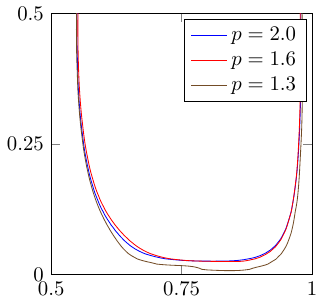}
    \caption{The dynamic Cheeger ratio of level sets of the first eigenfunction of $\dynamic{\Delta_p}$ for the transitory double gyre (top left), statistics of the ratio when choosing the level set randomly (top right) and the level set with the lowest Cheeger ratio for different $p$ (bottom left, closeup bottom right).   }
    \label{fig:double_gyre_statistics}
\end{figure}

\subsection{The cylinder flow}
\label{sec:exp_cylinder}

As a more complex example we look at the flow of the vector field 
\begin{align*}
    \dot x(t) &= c - A(t) \sin(x-\nu t) \cos(y) + \varepsilon \Gamma(g(x,y,t))\sin(t/2)\\
    \dot y(t) &= A(t)\cos(x-\nu t) \sin(y)
\end{align*}
on the domain $M = 2\pi S^1 \times [0,\pi]$, where $A(t) = 1 + \sin(2\sqrt{5}t)/8$, $\Gamma(\psi) = 1/(\psi^2+1)^2$, $g(x,y,t) = \sin(x-\nu t) \sin(y) +y/2-\pi/4$, $c=0.5$, $\nu=0.5$, and $\varepsilon=0.25$, cf.~\cite[Section B]{froyland_fast_2015}.

As the flow time we choose $T=40$; the time integration is performed  as in the transitory double gyre example.
The dynamic $p$-Laplacian is discretised on a $200\times 100$ grid. Its leading eigenfunction and its level sets are constructed as in the previous examples (see \cref{sec:exp_static} for details). As the domain is 
periodic in one coordinate, the shoelace formula for the area
of a polygon does not hold anymore, and we instead determine the area  of a superlevel set by
counting the number of vertices of a regular $200\times 100$ grid that lie in 
it.

The leading eigenfunctions $u^{(1)}_p$ of $\dynamic{\Delta_p}$ for $p\in \{2.0, 1.6, 1.3\}$ are depicted in \Cref{fig:cylinder_flow_eigfunction}, and the dynamic 
Cheeger ratios of the superlevel sets of the $u^{(1)}_p$ and their statistics are shown
in \Cref{fig:cylinder_flow_ratios}. For the respective
superlevel sets with minimal dynamic Cheeger ratio see 
\Cref{fig:cylinder_flow_bestsets}.
\begin{figure}[H]
\begin{tabular}{ c c c }
\includegraphics[width=5cm, trim={0 0 1.5cm 0}, clip]{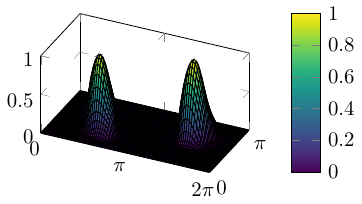} &
\includegraphics[width=5cm, trim={0 0 1.5cm 0}, clip]{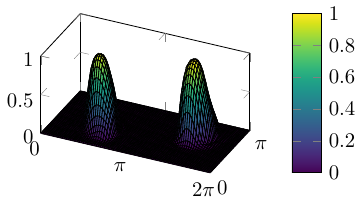} &
\includegraphics[width=5cm, trim={0 0 1.5cm 0}]{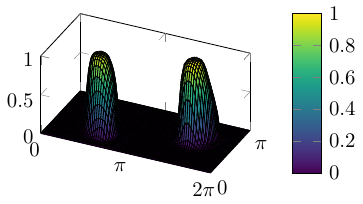}\\
\includegraphics[width=4cm]{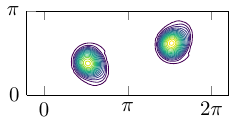} &
\includegraphics[width=4cm]{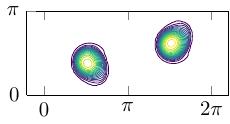} &
\includegraphics[width=4cm]{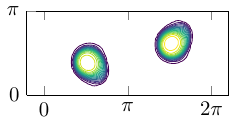}
\end{tabular}
\caption{
First eigenfunction of $\dynamic{\Delta}_p$ for
the cylinder flow and 
$p=2.0$ (left), $p=1.6$ (centre) and $p=1.3$ (right).}
\label{fig:cylinder_flow_eigfunction}
\end{figure}

\begin{figure}[H]
    \centering
    \includegraphics[height=5.2cm]{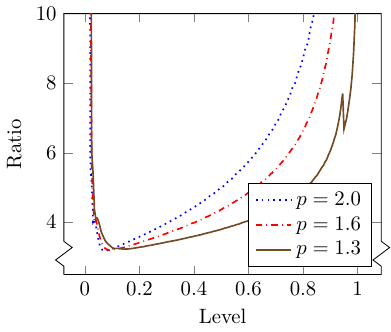}
    \hfil
    \includegraphics[height=5cm]{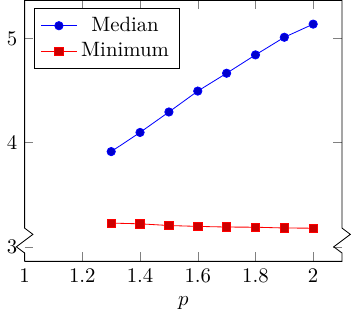}
    \caption{
        Dynamic Cheeger ratio of level sets of the first eigenfunction (left) and statistics of the ratio when choosing the level set randomly (right) 
        for the cylinder flow.}
        \label{fig:cylinder_flow_ratios}
\end{figure}

\begin{figure}[H]
    \centering
    \includegraphics[height=4cm]{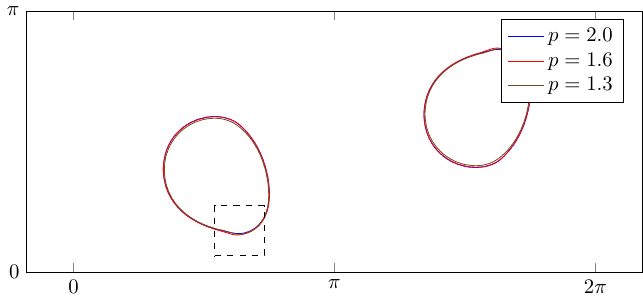}
    \includegraphics[height=4cm]{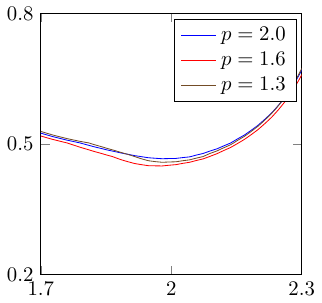}
    \caption{
    Level set with the smallest Cheeger ratio for different $p$ for the cylinder flow (on the right a closeup of the region indicated in the left plot).}
        \label{fig:cylinder_flow_bestsets}
\end{figure}

\subsection{The standard map}
\label{sec:exp_standard_map}

As an exploratory example, we consider the standard map on the flat $2$-torus $\mathbb{T}^2$
\[
    T(x,y) = (x+y+a \sin(x), y+ a\sin(x)) \;\; (\mathrm{mod} \;2\pi)
\]
with $a = 0.971635$ as in~\cite{froyland_fast_2015}. Since the boundary is empty here, the first eigenfunction $u_p^{(1)}$ is constant.  We thus consider  the second eigenfunction $u_p^{(2)}$ and use \Cref{alg:yaozhousecond}  to compute it. 
In order to ensure solvability of  \cref{eq:definitionbard} we add a zero-mean condition $\int_M v = 0$ in step \ref{alg:yaozhousecond3a}. For convenience, we also add a zero-mean condition $\int_M \dynamic{d} = 0$ in order to ensure uniqueness of the solution~$\dynamic{d}$. 
For boundaryless domains, the relevant dynamic Cheeger constant is not (\ref{eq:dynCheeger}), but instead the quantity
\[
   \dynamic{h}_{\rm Neum}(M,T):=
   \inf_{A\subset M}
    \frac{\ell_{d-1}(\partial \subsetname) + \ell_{d-1}(\partial T\subsetname)}
         {2\min\{\ell_d(\subsetname), \ell_d(M\backslash \subsetname)\}}, 
\]
see\ \cite{froyland_dynamic_2015}.
We compute areas as in the previous example.
The low-dimensional optimisation problem that arises in the computation of $u^+(v)$ in \Cref{alg:yaozhousecond} was solved to local optimality with the L-BFGS method constrained to the unit sphere in $\mathbb{R}^2$, as described in \cite{absil2009optimization} and implemented in \texttt{Optim.jl} \cite{mogensen2018optim}. As recommended in \cite{yao_numerical_2007}, we use the $t_i$ from 
the last iteration to initialise the optimisation 
for the next iteration to promote continuity of $u^+$. 
Otherwise, the computations have been done as in 
\cref{sec:exp_static} and \cref{sec:exp_rot_gyre}.
As vertical translations and $T$ are volume-preserving and their Jacobian does not depend on $y$, $\dynamic{F}, G$ and thus $\dynamic J$ are invariant under vertical translations. To deal with the resulting non-uniqueness of eigenfunctions, we shift an eigenfunction such that its maximum has $y$-coordinate~$\pi$; this makes the  eigenfunctions more easily comparable. 

In \Cref{fig:standard_map_efuns}, the second eigenfunction is shown for different $p$. We again observe that the eigenfunction becomes flatter around its maximum as $p$ decreases.  Similarly to the previous examples, the statistics of the Cheeger ratio $\dynamic{h}_{\rm Neum}(M,T)$ are given in \Cref{fig:standard_map_statistics}.
\begin{figure}[H]
\begin{tabular}{ c c c }
\includegraphics[width=5cm, trim={0 0 1.5cm 0}, clip]{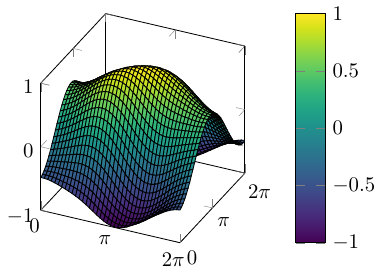} &
\includegraphics[width=5cm, trim={0 0 1.5cm 0}, clip]{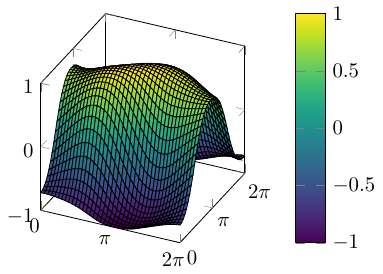} &
\includegraphics[width=5cm, trim={0 0 1.5cm 0} ]{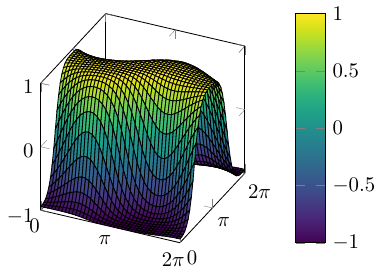}\\
\includegraphics[width=4cm]{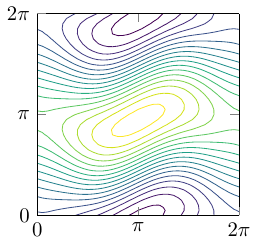} &
\includegraphics[width=4cm]{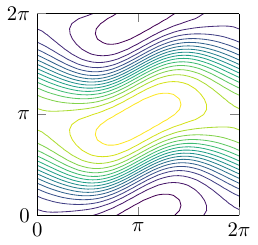} &
\includegraphics[width=4cm]{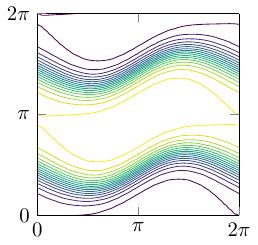}
\end{tabular}
\caption{Second eigenfunction of $\dynamic{\Delta}_p$ for the
standard map and
$p=2.0$ (left), $p=1.6$ (centre) and $p=1.3$ (right)}.
\label{fig:standard_map_efuns}
\end{figure}

\begin{figure}[H]
    \centering
    \includegraphics[width=0.32\textwidth]{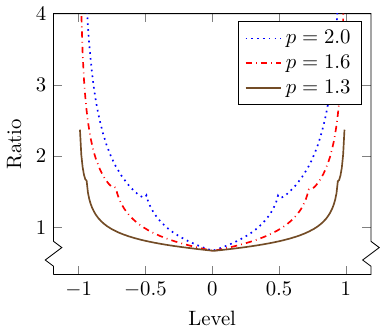}
    \includegraphics[width=0.31\textwidth]{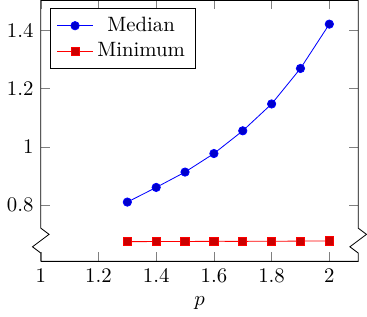}
    \includegraphics[width=0.28\textwidth]{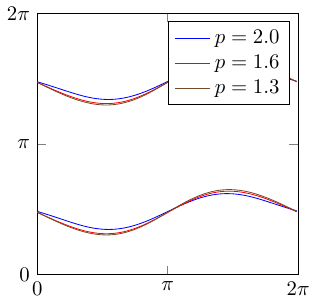}
    \caption{Dynamic Cheeger ratio of level sets of the second eigenfunction
    of $\dynamic{\Delta_p}$ for the standard map (left), statistics of the ratio when choosing the level set randomly (center) and the level set with the lowest Cheeger ratio for different $p$ (right).}
    \label{fig:standard_map_statistics}
\end{figure}

\section{Conclusions and outlook}

We introduced a dynamic version of the nonlinear $p$-Laplace operator -- denoted $\dynamic{\Delta}_p$ -- suitable for the spectral analysis of dynamical systems.
In \Cref{thm:dyn_eigenvalue} we proved the existence of a positive leading eigenvalue $\dynamic{\lambda}_p$ for $-\dynamic{\Delta}_p$ with homogeneous Dirichlet boundary conditions.
Associated to the dynamic spectral geometry of the dynamical system is the dynamic Cheeger constant $\dynamic{h}(M,T)$ defined in \eqref{eq:dynCheeger}, which quantifies how well the boundary of a carefully chosen subset in $M$ can resist growth under evolution of the dynamics.
In \Cref{thm:cheegerdynamic} we developed a Cheeger-type inequality, relating $\dynamic{\lambda}_p$
and the dynamic Cheeger constant for $p\ge 1$.
\Cref{thm:cheegerconv} showed that this inequality becomes increasingly tight as $p$ approaches 1, and in fact that~$\lim_{p\to 1}\lambda^D_p=h^D(M,T)$.

We then turned our attention to the numerical approximation of the dynamic $p$-Laplacian and its eigenfunctions. 
We proposed an algorithm -- modifying a similar algorithm for the classical $p$-Laplacian due to Yao \& Zhou~\cite{yao_numerical_2007} -- to solve the nonlinear eigenproblems for $1<p<2$, and compared the algorithm performance and solution quality for $p$ between 1 and 2 on a variety of examples. 
We found that (i) the level sets of the leading eigenfunctions tended to increasingly concentrate around the boundary of the optimal coherent set as $p\to 1$ and (ii) the optimal level set for the dynamic 2-Laplacian was often not very different from the optimal level set of the dynamic $p$-Laplacian for $p$ close to~1.
This first finding suggests that the ``sharper'' eigenfunctions of the dynamic $p$-Laplacian may be easier to use to find coherent sets than those of the dynamic $2$-Laplacian, including using multiple eigenfunctions post-processed via feature-separation algorithms such as~\cite{FrRoSa19}.
The second finding suggests that not too much is lost in terms of solution quality (where coherent set identification is quantified via $h^D(M,T)$) by using eigenfunctions of the dynamic 2-Laplacian.

There are numerous aspects of this work that could be addressed in the future. 
One practical motivation for the current work was that (first) eigenfunctions of $\Delta_p$ converge to characteristic functions of the Cheeger set. Although our results deliver positive empirical evidence, such a result for the dynamic $p$-Laplacian $\dynamic{\Delta}_p$ has yet to be proven.
Also, the uniqueness and positivity of the first eigenfunction of~$\dynamic{\Delta}_p$ is not yet established, as discussed in \Cref{rem:uniq}.

Finally, we have considered homogeneous Dirichlet boundary conditions, which is consistent with looking for sets $\subsetname \subset M$ in the isoperimetric problem \eqref{eq:cheeger} that have a compact closure in~$M$. 
Neumann boundary conditions, on the other hand, translate to problems where $\partial \subsetname$ can touch~$\partial M$. 
Results concerning Neumann boundary conditions for $\Delta_p$ are less well developed, but extensions of our results to the dynamic $p$-Laplacian under Neumann boundary conditions are natural to pursue.

\section*{Acknowledgements}
AD gratefully acknowledges support through the Bavarian Ministry of Science and Art. The research of GF is partially supported by an Australian Research Council Discovery Project, a Universities Australia 
 Australia-Germany Joint Research Cooperation Scheme, and an Einstein Visiting Fellowship to the Freie Universität Berlin. GF is grateful to the Departments of Mathematics at the Freie Universität Berlin and the University of Bayreuth for their generous hospitality. PK has been partially supported by the Deutsche Forschungsgemeinschaft (DFG) through grant CRC 1114 ``Scaling Cascades in Complex Systems'', Project Number 235221301, Project A08 ``Characterization and prediction of quasi-stationary atmospheric states’’. 

\appendix

\section{Appendix}

\subsection{On the min-max theorem}
\label{ssec:minmax}
\begin{thm}[Courant--Fischer]
\label{thm:Courant-Fischer}
Let $A\in \R^{d\times d}$ be symmetric.  The eigenvalues $\lambda^{(1)}\leq \lambda^{(2)} \leq \cdots \leq \lambda^{(d)}$ of $A$ are given by 
\[
\lambda^{(k)} = \min\left\{\max \left\{ \< u, Au\> \mid u\in V\cap S \right\} \mid V\in V_k \right\},
\]
where $V_k$ is the set of $k$-dimensional subspaces of $\R^d$ and $S$ is the unit sphere in $\R^d$.
\end{thm} 

\begin{prop}
\label{prop:CF_complement}
In the setting of Theorem~\ref{thm:Courant-Fischer}, if $u_k$ denotes the $\|\cdot\|_2$-normalised eigenvector at $\lambda^{(k)}$, $L_k=\spn\left(u_1,\ldots,u_k\right)=:[u_1,\ldots,u_k]$, and $L_k'$ some fixed complement of $L_k$, then
\[
\lambda^{(k)} = \min \left\{\max \left\{ \<u, Au\> \mid u \in [L_{k-1}, v] \cap S \right\} \mid v\in L_{k-1}' \right\}.
\]
Moreover, if $v_k$ denotes a minimiser of the above expression, then 
\[
u_k \in \arg\max \big\{ \<u, Au\> \mid u \in [L_{k-1}, v_k] \cap S \big\}.
\]
\end{prop}
\begin{proof}
    In the following computation, line by line, we make use of the following facts:
    \begin{enumerate}
    \setlength{\itemsep}{-3pt}
        \item The Courant--Fischer theorem.
        \item For $v\in L_{k-1}'$ one has $[L_{k-1},v] \in V_k$.  
        \item Estimating the minimum from above by setting $v = v_k \in L_{k-1}'$ such that $[L_{k-1},v_k] = L_k$. 
        \item Using that the eigenvectors $u_i$ are mutually orthogonal and setting $u = \sum_{i=1}^k t_i u_i$ with $\sum_{i=1}^k t_i^2=1$.
    \end{enumerate}
    The computation
    \begin{align*}
        \lambda^{(k)} &= \min\left\{\max \left\{ \< u, Au\> \mid u\in V\cap S \right\} \mid V\in V_k \right\} \\
        &\le \min \left\{\max \left\{ \<u, Au\> \mid u \in [L_{k-1}, v] \cap S \right\} \mid v\in L_{k-1}' \right\} \\
        &\le \max \left\{ \<u, Au\> \mid u \in L_k \cap S \right\} \\
        &= \max \left\{ \textstyle \sum_{i=1}^k \lambda^{(i)} t_i^2 \, \mid \, \sum_{i=1}^k t_i^2 = 1 \right\} \le \left\{ \textstyle \sum_{i=1}^k \lambda^{(k)} t_i^2 \, \mid \, \sum_{i=1}^k t_i^2=1 \right\} = \lambda^{(k)}
    \end{align*}
    yields the claim. Equality holds in the last inequality for $t_k=1$ and $t_i=0$ for $i\neq k$, showing that a maximiser is indeed~$u=u_k$.
\end{proof}

\subsection{A lemma for the proof of \torpdf{\Cref{thm:cheegerdynamic}}{}}

\begin{lem}\label{lem:introducingabs} Let $M\in \mathbb{R}^d$
    compact and $w\in C^{\infty}_0(M)$ and define
    \begin{align}
        A(t) := \{x\mid w(x) > t\} \\
        B(t) := \{x\mid |w(x)| >t \}.
    \end{align}
    \begin{enumerate}[a)]
        \item For almost all $t$ we have $\partial A(t) = w^{-1}(t)$
        \label{eq:borderispreimage}
        \item For almost all $t$ we have
            $\partial B(t) = \partial A(t) \, \dot\cup \, \partial A(-t)$
        \item The following integrals coincide:
    \[
        \int_{-\infty}^\infty \ell_{d-1}(\partial A(t))\ dt=
        \int_0^\infty \ell_{d-1}(\partial B(t))dt.
    \]
    \end{enumerate}
    In all statements the boundary is taken within $M$.
\end{lem}

\begin{proof}
    The inclusion $\partial A(t) \subseteq w^{-1}(t)$ follows
    for all $t$ by continuity of $w$.
    On the other hand let $x\in w^{-1}(t)$. We may assume $\nabla w(x) \neq 0$
    as by Sard's theorem $w^{-1}(t)$ contains no critical points
    for almost all $t$.
    We may also assume $t\neq 0$ (as $\{0\}$ is a
    null set), implying $x\in M^\circ$ and
    in particular $x+s\nabla w(x)\in M$ for small $s$.
    Now
    \begin{equation}
         w(x+s\nabla w(x)) - w(x) = s|\nabla w(x)|^2 + \mathcal{O}(s^2)
        \label{eq:taylorargument}
    \end{equation}
    is positive for small
    $s>0$ and negative for small $s<0$ and it follows that
    $x\in \partial A(t)$. Hence, $\partial A(t) \supseteq w^{-1}(t)$
    is fulfilled, too and we have shown (a).
    Now this implies that for almost all $t$ we have
    \begin{equation}
        \partial B(t) = \partial A(t) \, \dot\cup \, \partial A(-t)
        \label{eq:disjointunion}
    \end{equation}
    in the following way: if $x\in \partial B(t)$ then, again by continuity of
    $w$ we must have $w(x) = \pm t$, which for almost
    all $t$ implies $x\in \partial A(t) \cup \partial A(-t)$
    by (\ref{eq:borderispreimage}). Conversely, if
    $x \in\partial A(\pm t)$ then almost surely $x\in\partial B(t)$,
    as there are values bigger and smaller than $w(x)$
    attained arbitrarily close to $x$ by the same construction as
    in \eqref{eq:taylorargument}. Finally, the sets $w^{-1}(t)$ and
    $w^{-1}(-t)$ are obviously disjoint for $t\neq 0$ and
    so are the sets $\partial A(t)$ and $\partial A(-t)$ for
    almost all $t$.
    This finishes the proof of (b).
    Finally, we can use this to conclude
    \begin{align}
        \int_0^\infty \ell_{d-1}(\partial B(t))dt &=
         \int_0^\infty \ell_{d-1}(\partial A(t)) + \ell_{d-1}(\partial A(-t))\ dt \\
        &= \int_0^\infty \ell_{d-1}(\partial A(t)) +
            \int_{-\infty}^0 \ell_{d-1}(\partial A(t))\ dt \\
        &= \int_{-\infty}^\infty \ell_{d-1}(\partial A(t)).
     \end{align}
\end{proof}

\bibliography{main}

\begin{thebibliography}{10}

\bibitem{absil2009optimization}
P.-A. Absil, R.~Mahony, and R.~Sepulchre.
\newblock Optimization algorithms on matrix manifolds.
\newblock In {\em Optimization Algorithms on Matrix Manifolds}. Princeton
  University Press, 2009.

\bibitem{adams2003sobolev}
R.~A. Adams and J.~J. Fournier.
\newblock {\em Sobolev spaces}.
\newblock Elsevier, 2003.

\bibitem{aref1984}
H.~Aref.
\newblock Stirring by chaotic advection.
\newblock {\em Journal of Fluid Mechanics}, 143:1--21, 1984.

\bibitem{aref2017}
H.~Aref, J.~R. Blake, M.~Budi{\v{s}}i{\'c}, S.~S. Cardoso, J.~H. Cartwright,
  H.~J. Clercx, K.~El~Omari, U.~Feudel, R.~Golestanian, E.~Gouillart, G.~F. van
  Heijst, T.~S. Krasnopolskaya, Y.~L. Guer, R.~S. MacKay, V.~V. Meleshko,
  G.~Metcalfe, I.~Mezić, A.~P.~S. de~Moura, O.~Piro, M.~F.~M. Speetjens,
  R.~Sturman, J.-L. Thiffeault, and I.~Tuval.
\newblock Frontiers of chaotic advection.
\newblock {\em Reviews of Modern Physics}, 89(2):025007, 2017.

\bibitem{Badia2020}
S.~Badia and F.~Verdugo.
\newblock Gridap: An extensible finite element toolbox in {Julia}.
\newblock {\em Journal of Open Source Software}, 5(52):2520, 2020.

\bibitem{badiale_semilinear_2011}
M.~Badiale and E.~Serra.
\newblock {\em Semilinear {Elliptic} {Equations} for {Beginners}}.
\newblock Springer, London, 2011.

\bibitem{BaKo17}
R.~Banisch and P.~Koltai.
\newblock Understanding the geometry of transport: {D}iffusion maps for
  {L}agrangian trajectory data unravel coherent sets.
\newblock {\em Chaos}, 27(3):035804, 2017.

\bibitem{chavel1984eigenvalues}
I.~Chavel.
\newblock {\em Eigenvalues in {R}iemannian geometry}.
\newblock Academic Press, 1984.

\bibitem{chavel_isoperimetric_2001}
I.~Chavel.
\newblock {\em Isoperimetric Inequalities: Differential Geometric and Analytic
  Perspectives}.
\newblock Number 145 in Cambridge tracts in mathematics. Cambridge University
  Press, Cambridge ; New York, 2001.

\bibitem{gunning_lower_1970}
J.~Cheeger.
\newblock A {Lower} {Bound} for the {Smallest} {Eigenvalue} of the {Laplacian}.
\newblock In R.~C. Gunning, editor, {\em Problems in {Analysis}: {A}
  {Symposium} in {Honor} of {Salomon} {Bochner} ({PMS}-31)}, pages 195--200.
  Princeton University Press, 1970.

\bibitem{Contour}
D.~Darakananda and T.~Lycke.
\newblock {C}ontour.jl.
\newblock \url{https://github.com/JuliaGeometry/Contour.jl}.
\newblock Accessed: June 2, 2023.

\bibitem{federer1960normal}
H.~Federer and W.~H. Fleming.
\newblock Normal and integral currents.
\newblock {\em Annals of Mathematics}, pages 458--520, 1960.

\bibitem{feng_analysis_2003}
X.~Feng and A.~Prohl.
\newblock Analysis of total variation flow and its finite element
  approximations.
\newblock {\em ESAIM: Mathematical Modelling and Numerical Analysis},
  37(3):533--556, 2003.

\bibitem{froyland_dynamic_2015}
G.~Froyland.
\newblock Dynamic isoperimetry and the geometry of {Lagrangian} coherent
  structures.
\newblock {\em Nonlinearity}, 28(10):3587--3622, 2015.

\bibitem{froyland_fast_2015}
G.~Froyland and O.~Junge.
\newblock On fast computation of finite-time coherent sets using radial basis
  functions.
\newblock {\em Chaos: An Interdisciplinary Journal of Nonlinear Science},
  25(8):087409, 2015.

\bibitem{froyland_robust_2018}
G.~Froyland and O.~Junge.
\newblock Robust {FEM}-{based} {extraction} of {finite}-{time} {coherent}
  {sets} {using} {scattered}, {sparse}, and {incomplete} {trajectories}.
\newblock {\em SIAM Journal on Applied Dynamical Systems}, 17(2):1891--1924,
  2018.

\bibitem{FrRoSa19}
G.~Froyland, C.~P. Rock, and K.~Sakellariou.
\newblock Sparse eigenbasis approximation: Multiple feature extraction across
  spatiotemporal scales with application to coherent set identification.
\newblock {\em Communications in Nonlinear Science and Numerical Simulation},
  77:81--107, 2019.

\bibitem{haller2013coherent}
G.~Haller and F.~J. Beron-Vera.
\newblock Coherent {L}agrangian vortices: The black holes of turbulence.
\newblock {\em Journal of Fluid Mechanics}, 731:R4, 2013.

\bibitem{horak_numerical_2011}
J.~Horák.
\newblock Numerical investigation of the smallest eigenvalues of the
  $p$-{Laplace} operator on planar domains.
\newblock {\em Electronic Journal of Differential Equations}, 2011(132):1--30,
  2011.

\bibitem{kawohl_isoperimetric_2003}
B.~Kawohl and V.~Fridman.
\newblock Isoperimetric estimates for the ﬁrst eigenvalue of the
  $p$-{Laplace} operator and the {Cheeger} constant.
\newblock {\em Commentationes Mathematicae Universitatis Carolinae},
  44:659--667, 2003.

\bibitem{lefton_numerical_1997}
L.~Lefton and D.~Wei.
\newblock Numerical approximation of the first eigenpair of the $p$-{L}aplacian
  using finite elements and the penalty method.
\newblock {\em Numerical Functional Analysis and Optimization},
  18(3-4):389--399, 1997.

\bibitem{leonardi2015overview}
G.~P. Leonardi.
\newblock An overview on the {Ch}eeger problem.
\newblock In {\em New trends in shape optimization}, volume 166, pages
  117--139. Birkh\"auser, Cham, 2015.

\bibitem{lindqvist1990equation}
P.~Lindqvist.
\newblock On the equation $\operatorname{div}(|\nabla u|^{p-2}\nabla u) +
  \lambda|u|^{p-2}u = 0$.
\newblock {\em Proceedings of the American Mathematical Society}, pages
  157--164, 1990.

\bibitem{lindqvist_nonlinear_2008}
P.~Lindqvist.
\newblock A nonlinear eigenvalue problem.
\newblock {\em Topics in mathematical analysis}, 3:175--203, 2008.

\bibitem{maz1960classes}
V.~G. Maz'ya.
\newblock Classes of domains and imbedding theorems for function spaces.
\newblock In {\em Doklady Akademii Nauk}, volume 133, pages 527--530. Russian
  Academy of Sciences, 1960.

\bibitem{mogensen2018optim}
P.~K. Mogensen and A.~N. Riseth.
\newblock Optim: A mathematical optimization package for {Julia}.
\newblock {\em Journal of Open Source Software}, 3(24):615, 2018.

\bibitem{mosovsky_transport_2011}
B.~A. Mosovsky and J.~D. Meiss.
\newblock Transport in transitory dynamical systems.
\newblock {\em SIAM Journal on Applied Dynamical Systems}, 10(1):35--65, 2011.

\bibitem{parini_second_2010}
E.~Parini.
\newblock The second eigenvalue of the $p$-{L}aplacian as $p$ goes to~1.
\newblock {\em International Journal of Differential Equations}, 2010:1--23,
  2010.

\bibitem{parini_introduction_2011}
E.~Parini.
\newblock An introduction to the {Cheeger} {Problem}.
\newblock {\em Surv. Math. Appl.}, 6:9--21, 2011.

\bibitem{rackauckas2017differentialequations}
C.~Rackauckas and Q.~Nie.
\newblock Differentialequations.jl--a performant and feature-rich ecosystem for
  solving differential equations in {Julia}.
\newblock {\em Journal of Open Research Software}, 5(1), 2017.

\bibitem{RevelsLubinPapamarkou2016}
J.~{Revels}, M.~{Lubin}, and T.~{Papamarkou}.
\newblock Forward-mode automatic differentiation in {J}ulia.
\newblock {\em arXiv:1607.07892 [cs.MS]}, 2016.

\bibitem{rade_mathematics_2004}
L.~Råde and B.~Westergren.
\newblock {\em Mathematics {Handbook} for {Science} and {Engineering}}.
\newblock Springer Berlin Heidelberg, Berlin, Heidelberg, 2004.

\bibitem{simader1972dirichlet}
C.~G. Simader.
\newblock On {D}irichlet's {B}oundary {V}alue {P}roblem.
\newblock Lecture Notes in Mathematics, Vol. 268, 1972.

\bibitem{StrangFix:73}
G.~Strang and G.~J. Fix.
\newblock An analysis of the finite element method.
\newblock Prentice-{Hall} {Series} in {Automatic} {Computation}. {Englewood}
  {Cliffs}, {N}.{J}.: {Prentice}-{Hall}, {Inc}. {XIV}, 306, 1973.

\bibitem{sturman2006}
R.~Sturman, J.~M. Ottino, and S.~Wiggins.
\newblock {\em The mathematical foundations of mixing: the linked twist map as
  a paradigm in applications: micro to macro, fluids to solids}, volume~22.
\newblock Cambridge University Press, 2006.

\bibitem{yao_numerical_2007}
X.~Yao and J.~Zhou.
\newblock Numerical {methods} for {computing} {nonlinear} {eigenpairs}: {Part}
  {I}. {Iso}-{homogeneous} {cases}.
\newblock {\em SIAM Journal on Scientific Computing}, 29(4):1355--1374, 2007.

\end{thebibliography}
\bibliographystyle{abbrv}

\end{document}